\documentclass{amsart}
\usepackage{amssymb}
\usepackage{amsmath}
\usepackage{amscd}
\usepackage{url}

\newtheorem{prop}{Proposition}[section]

\newtheorem{conj}[prop]{Conjecture}

\newtheorem{lem}[prop]{Lemma}

\theoremstyle{definition}

\def\id{\mathop{\mathrm{id}}\nolimits}

\newcommand{\ord}{{\mathrm {ord}}}

\newcommand{\Hom}{{\mathrm {Hom}}}

\newcommand{\Ind}{{\mathrm {Ind}}}

\newcommand{\Fil}{{\mathrm {Fil}}}

\newcommand{\tr}{{\mathrm {tr}}}
\newcommand{\Tr}{{\mathrm {Tr}}}

\newcommand{\Sym}{{\mathrm {Sym}}}

\newcommand{\st}{{\mathrm {st}}}
\newcommand{\dR}{{\mathrm {dR}}}
\newcommand{\op}{{\mathrm {op}}}

\newcommand{\Frob}{{\mathrm {Frob}}}
\def\rank{\mathop{\mathrm{ rank}}\nolimits}
\newcommand{\Gal}{\mathrm {Gal}}

\newcommand{\ad}{{\mathrm {ad}}}

\newcommand{\diag}{\mathrm{diag}}

\newcommand{\wt}{{\mathrm {wt}}}
\newcommand{\A}{{\mathbb A}}
\newcommand{\CC}{{\mathbb C}}
\newcommand{\C}{{\mathbb C}}
\newcommand{\RR}{{\mathbb R}}
\newcommand{\R}{{\mathbb R}}

\newcommand{\QQ}{{\mathbb Q}}
\newcommand{\Q}{{\mathbb Q}}
\newcommand{\ZZ}{{\mathbb Z}}

\newcommand{\MMM}{{\mathcal M}}
\newcommand{\SSS}{{\mathcal S}}

\newcommand{\h}{{\mathcal H}}
\newcommand{\VVV}{{\mathbb V}}

\newcommand{\g}{{\mathfrak g}}
\newcommand{\n}{{\mathfrak n}}
\newcommand{\hh}{{\mathfrak h}}

\newcommand{\q}{{\mathfrak q}}

\newcommand{\mm}{{\mathfrak m}}

\newcommand{\FF}{{\mathbb F}}

\newcommand{\GL}{\mathrm {GL}}

\newcommand{\SL}{\mathrm {SL}}
\newcommand{\Sp}{\mathrm {Sp}}
\newcommand{\SO}{\mathrm {SO}}
\newcommand{\GSp}{\mathrm {GSp}}
\newcommand{\PGSp}{\mathrm {PGSp}}

\newcommand{\Qbar}{\overline{\mathbb Q}}

\newcommand{\rhobar}{\overline{\rho}}
\newcommand{\rhobart}{\overline{\tilde{\rho}}}
\newcommand{\rhot}{\tilde{\rho}}

\begin{document}
\title{Eisenstein congruences for split reductive groups}
\author{Jonas Bergstr\"om}
\author{Neil Dummigan}
\date{October 13th, 2015.}
\address{Matematiska institutionen\\ Stockholms universitet\\ 106 91 Stockholm\\Sweden.}
\email{jonasb@math.su.se}
\address{University of Sheffield\\ School of Mathematics and Statistics\\
Hicks Building\\ Hounsfield Road\\ Sheffield, S3 7RH\\
U.K.}
\email{n.p.dummigan@shef.ac.uk}

\begin{abstract}
We present a general conjecture on congruences between Hecke eigenvalues of parabolically induced and cuspidal automorphic representations of split reductive groups, modulo divisors of critical values of certain $L$-functions. We examine the consequences in several special cases, and use the Bloch-Kato conjecture to further motivate a belief in the congruences.
\end{abstract}

\maketitle

\section{Introduction}
Ramanujan discovered the congruence $\tau(p)\equiv 1+p^{11}\pmod{691}$, for all primes $p$, where $\Delta=\sum_{n=1}^{\infty}\tau(n)q^n=q\prod_{n=1}^{\infty}(1-q^n)^{24}$. We may view this as being a congruence between Hecke eigenvalues, for $T(p)$ acting on the cusp form $\Delta$ of weight~$12$ for $\SL_2(\ZZ)$, and on the Eisenstein series $E_{12}$ of weight~$12$. The modulus $691$ comes from a certain $L$-function evaluated at a critical point depending on the weight; specifically it divides the numerator of the rational number $\frac{\zeta(12)}{\pi^{12}}$. Conjecture \ref{main} in this paper is a very wide generalisation of Ramanujan's congruence, to congruences between Hecke eigenvalues of automorphic representations of $G(\A)$, where $\A$ is the adele ring and $G/\Q$ is any connected, reductive group. (Here we look only at the case that $G$ is split, but we shall return to the non-split case elsewhere.) On one side of the congruence is a cuspidal automorphic representation $\tilde{\Pi}$ of $G$. On the other is one induced from a cuspidal automorphic representation $\Pi$ of the Levi subgroup $M$ of a maximal parabolic subgroup $P$.
The modulus of the congruence comes from a critical value of a certain $L$-function, associated to $\Pi$ and to the adjoint representation of the $L$-group $\hat{M}$ on the Lie algebra $\hat{\n}$ of the unipotent radical of the maximal parabolic subgroup $\hat{P}$ of $\hat{G}$. Starting from $\Pi$, we conjecture the existence of $\tilde{\Pi}$, satisfying the congruence. Ramanujan's congruence is an instance of the case $G=\GL_2, M=\GL_1\times\GL_1$.

For an odd prime $q$, and even $k$ such that $2\leq k\leq q-3$ and with $q$ dividing $\zeta(k)/\pi^k,$ Ribet exploited a congruence of this type (still $G=\GL_2, M=\GL_1\times\GL_1$) to construct an element of order $q$ in the class group of the cyclotomic field $\QQ(\xi_q)$, more precisely in the $\chi^{1-k}$-eigenspace for the action of $\Gal(\QQ(\xi_{q})/\QQ)$, where $\chi$ is the cyclotomic character \cite{Ri}. The Hecke eigenvalues for a cusp form $f$ are traces of Frobenius for the $2$-dimensional $\q$-adic Galois representation. The congruence can be interpreted as a reducibility of this Galois representation modulo $\q$, with $1$-dimensional composition factors $\id$ and $\chi^{1-k}$ in a non-split extension which gives the element of the class group. (For technical reasons, Ribet replaced modular forms of weight $k$ and level $1$ by modular forms of weight $2$ and level $q$, with non-trivial character.) This element may be thought of as belonging to a Bloch-Kato Selmer group associated to $\zeta(k)$, confirming a prediction of the Bloch-Kato conjecture on special values of $L$-functions, given that  $q$ divides $\zeta(k)/\pi^k$. In fact, when Bloch and Kato \cite{BK} proved most of their conjecture in the case of the Riemann zeta function, the main ingredient was the Mazur-Wiles theorem \cite{MW} (Iwasawa's main conjecture), whose proof was a further development of Ribet's idea. In \S 14, we try to motivate the conjectured congruence, and in particular to justify the specific choice of $L$-value from which the modulus is extracted, by generalising Ribet's construction, to link the congruence to the Bloch-Kato conjecture. Though we cannot actually prove much, the adjoint action of $\hat{M}$ on $\hat{\n}$ appears in a plausible manner.

Beyond the case $G=\GL_2$ (and closely related congruences of Hilbert modular forms), maybe the next to be studied was that of $G=\GSp_2$, with $P$ the Klingen parabolic, i.e. congruences between genus-$2$ Klingen-Eisenstein series and cusp forms. The first example, found by Kurokawa \cite{Ku}, was a congruence mod $71^2$ between the Hecke eigenvalues of two scalar-valued Siegel modular forms of genus-$2$ and weight $20$, one a cusp form, the other the Klingen-Eisenstein series of the unique normalised genus-$1$ cusp form $\Delta_{20}$ of level $1$ and weight $20$. The modulus comes from the rightmost critical value $L(\Sym^2 \Delta_{20},38)$. For the genus-$2$ Hecke operator traditionally called $T(p)$, the Hecke eigenvalue for the Klingen-Eisenstein series, i.e. the right-hand-side of the congruence, is $a_p(\Delta_{20})(1+p^{k-2})$, where $k=20$ and $\Delta_{20}=\sum_{n=1}^{\infty}a_p(\Delta_{20})q^n$.

Congruences of this type, but for $\Lambda$-adic forms, where $\Lambda$ is a two-variable Iwasawa algebra, and the modulus is a $q$-adic adjoint $L$-function, were proved by Urban \cite{U2}, and used by him to prove that that $q$-adic $L$-function divides the characteristic ideal to which the main conjecture says that it should be equal \cite{U1}. Skinner and Urban have similarly used the non-split case $G=GU(2,2)$, $P$ a Klingen parabolic, in their work on the main conjecture for $\GL_2$, see \cite{SU}. Both these works use an adaptation of Ribet's construction.

The case $G=\GSp_2$, $P$ the Siegel parabolic, arises out of work of Harder \cite[3.1]{H3}, and the first computational evidence was observed by him \cite{H1}, using computations of Hecke eigenvalues by Faber and van der Geer \cite{FvdG, vdG}. Their method involves computing the zeta function of curves whose Jacobians make up the mod $p$ points of $\mathcal A_2$, the moduli space of principally polarized abelian surfaces, which is a Siegel modular threefold. In their paper they computed Hecke eigenvalues for $p\leq 37$, for weights $j$, $k$ such that the space of genus-2 cusp forms is $1$-dimensional, assuming a conjecture on the endoscopic contribution to the cohomology of local systems on $\mathcal A_2$. This conjecture has been proven by Petersen \cite{P} (see also Weissauer \cite{We}), building on research of many people on the automorphic representations of $\GSp_2$. Here, the right hand side of the congruence is, for $T(p)$, of the form $a_p(f)+p^{k-2}+p^{j+k-1}$, where $f$ is a genus-$1$ cusp form of weight $j+2k-2$, and the modulus comes from the critical value $L(f,j+k)$.
The genus-$2$ cusp form whose Hecke eigenvalues appear on the left-hand-side is vector-valued of weight $\Sym^j\otimes\det^k$.

Harder himself generalised his conjecture to the case $G=\GSp_3$, $M\simeq \GL_2\times\GL_2$ \cite{H5}. Here the right-hand-side involves two genus-$1$ cusp forms, and there are two cases depending on the relative sizes of their weights.
Computations by the first named author, Faber and van der Geer, using the same techniques as mentioned above but for the Siegel modular six-fold $\mathcal A_3$, gave (assuming a conjecture on the endoscopic contribution to the cohomology of local systems on $\mathcal A_3$) Hecke eigenvalues for genus-3 vector-valued Siegel modular forms for $p \leq 17$, and aided by $L$-value approximations by Mellit, they produced numerical evidence for these conjectures \cite{BFvdG1}. They also found some apparent congruences in the case $G=\GSp_3$, $M\simeq \GL_1\times\GSp_2$, but did not put forward a guess for what $L$-value the modulus comes from.

In \S 2 we introduce some notation and basic facts on reductive groups, characters and cocharacters, automorphic representations, Satake parameters and infinitesimal characters. In \S 3 we look at the $L$-functions mentioned above, connected with the adjoint action of $\hat{M}$ on $\hat{\n}$. The first contribution of this paper is the discussion, towards the end of \S 3, of the relation between criticality of values of these $L$-functions and dominance of certain characters related to induced representations, explaining what might otherwise seem like a strange coincidence. In \S 4, after introducing what we need on the Bloch-Kato conjecture, we state the main conjecture on congruences.

In \S\S 5,6 and 7, we examine the cases $G=\GL_2, M\simeq \GL_1\times\GL_1$, $G=\GSp_2$, $P$ the Klingen parabolic, and $G=\GSp_2$, $P$ the Siegel parabolic, respectively. Hopefully it is evident already in these cases how efficiently our presentation leads, in a unified way, to the explicit determination of the right-hand-sides of congruences, which $L$-values the moduli come from, and the ``weights'' of the various objects involved. Another special feature is that, via the Bloch-Kato conjecture in the case of values of $L$-functions with missing Euler factors, we find a natural home for Harder's congruences ``of local origin'' \cite{H2}, and present examples in the new case of $G=\GSp_2$, $P$ the Klingen parabolic.

For $G=\GSp_3$ there are three maximal parabolics (up to conjugacy). In \S 8 we consider the case $M\simeq \GL_2\times\GL_2$, recovering the conjecture of Harder mentioned above. In \S 10 we look at $M\simeq \GL_1\times\GL_3$, for which we have no computational evidence. More interesting perhaps is the remaining case $M\simeq \GL_1\times\GSp_2$, covered in \S 9. Here we recover the conjectural congruences of which the first-named author, Faber and van der Geer found examples, and effortlessly arrive at critical values of a genus-$2$ standard $L$-function as the source of the modulus for the congruence. Showing, in special cases, that the primes, for which they found congruences, really do occur in the standard $L$-values, calls on earlier work which the second-named author never expected to lead anywhere further \cite{Du3,DIK}. This is connected with a quite different construction of elements in Bloch-Kato Selmer groups, related to the ``visibility'' construction of Cremona and Mazur \cite{CM}.

The spinor $L$-function of a genus-$2$ cusp form is involved in the case $G=\SO(4,3)$, $M\simeq \GL_1\times\SO(3,2)$, which we do not mention again in this paper, but return to in \cite{BDM}.

In \S\S 11 and 12, we consider the two conjugacy classes of maximal parabolics for $G$ the Chevalley group of type $G_2$, expecting that something interesting and testable might happen, since in both cases $M\simeq \GL_2$. What we observe, using the conjectural Gross-Savin functorial lift from $G_2$ to $\GSp_3$, is a remarkable consistency with the earlier cases involving $\GSp_3$ (and also $\GSp_2$). In the last part of \S 12, we also find something new. Suppose that $a\geq b\geq c>0$, with $c$ even and $a=b+c$. If $f$ is a genus-$1$ cuspidal Hecke eigenform of weight $a+b+6$ then we expect a congruence involving a genus-$3$ cuspidal Hecke eigenform of weight $(a-b,b-c,c+4)$, with right-hand-side $(1+p^{c+1})(a_p(f)+p^{b+2}+p^{a+3})$, and modulus coming from the critical value $L(f,a+c+5)$. Using Hecke eigenvalues calculated as in \cite{BFvdG1}, we checked numerical evidence for such a congruence in the case $(a,b,c)=(10,8,2)$. The condition $a=b+c$ is necessary for the congruence to come from $G=G_2$ via the Gross-Savin lift, but we noticed that it appears to work even without that condition, and found sixteen more examples, for which $a\neq b+c$.

The induced representations we consider depend on a parameter $s>0$ which is typically confined to $\ZZ$ or to $\frac{1}{2}+\ZZ$ (and is often bounded above too). We have to exclude $s=\frac{1}{2}$ or $1$ from the scope of Conjecture \ref{main}. But congruences of Hecke eigenvalues between CAP (cuspidal associated to a parabolic) and non-CAP cuspidal automorphic representations sometimes appear as a substitute in the case $s=\frac{1}{2}$ or $1$. See the remarks on Saito-Kurokawa lifts in \S 7 and Ikeda-Miyawaki lifts in \S 8. It seems possible that we can include
$s=\frac{1}{2}$ or $1$ at the expense of enlarging the set of ramified primes, see the remarks at the end of \S 5, and again in \S 7.

An important feature in the work of Harder is the cohomology of local systems on arithmetic quotients of locally symmetric spaces, and in \S 13 we work out the precise relationship between our way of arriving at conjectural congruences and Harder's. Another key aspect of his approach is the occurrence of the $L$-value in the denominator of a constant term of a generalised Eisenstein series (by a theorem of Langlands/Gindikin-Karpelevich \cite[Theorems 5.3,6.7]{Ki}). This too seems to be important for a fuller understanding, and affects the precise formulation and scope of the conjecture. As he has pointed out, the periods we divide by to normalise $L$-values are motivic in nature, whereas the periods he divides by to normalise the ratios of consecutive $L$-values appearing in constant terms are topological in nature.

In \S 14 we indicate how from a congruence of the type considered here, the existence of a non-zero element in a Bloch-Kato Selmer group ought to follow (though our argument is far from a proof). Then, according to the Bloch-Kato conjecture, we should find the modulus dividing the appropriate normalised $L$-value. Much of the numerical evidence we give or refer to goes in this direction, in that we look for congruences first, then having found them, confirm the divisibility of the $L$-value, as ``predicted'' by the Bloch-Kato conjecture. However, our Conjecture \ref{main} as presented goes in the other direction, predicting that given divisibility of an $L$-value, a congruence should follow. In other words, when the Bloch-Kato conjecture (applied to the $L$-values we look at here) predicts the existence of a non-zero element in a Selmer group, this element should be constructible from a congruence. One might ask what justification we have for such a conjecture, with this direction of implication.

First, in those few cases where anything is actually proved, e.g. for $G=\GL_2$, $M=\GL_1\times\GL_1$, or in cases for $G=\GSp_2$ and $P$ the Klingen parabolic, one starts from divisibility of an $L$-value and proves a congruence. When $G=\GSp_2$ and $P$ is the Siegel parabolic (Harder's conjecture), van der Geer looked at all level $1$ examples where the relevant spaces of cusp forms of genus $1$ and genus $2$ are $1$-dimensional. In all cases where a large enough prime divided the normalised $L$-value he {\em then} verified the expected congruence for $p\leq 37$ \cite[\S 27]{vdG}. In \S 12, where $G=G_2$, in the example with a congruence for $q=179$, one of us found this divisor of the $L$-value first, predicting a congruence which was then verified for $p\leq 17$ by the other one. In the non-split case $G=U(2,2)$ with $P$ the Siegel parabolic, the second-named author calculated an $L$-value first, finding divisors for $q=19$ and $q=37$, {\em then} computed some Hecke eigenvalues which turned out to be consistent with the expected congruences \cite{Du4}. In another non-split case, $G=U(2,1)$, the first-named author found apparent congruences for $q=53$ and $q=271$, using Hecke eigenvalues computed by him and van der Geer, {\em after} Harder had predicted them on the basis of $L$-value computations (see \cite{Du4}).

The authors met at the Max Planck Institute in Bonn in February 2010, and are grateful for the opportunity so provided. There they also attended a seminar by G. Harder, and participated in valuable discussions with him, continued on later occasions. They benefitted also from his comments on an earlier version of this paper. We were directed to \cite{Ki} and \cite{BG} by G. Harder and T. Berger, respectively.

\section{Induced representations}
For basic notions on reductive groups and automorphic representations, see \cite{Sp,BJ}, and associated articles in the same volume. Another useful reference is \cite{Ki}.

Let $G/\QQ$ be a connected, reductive algebraic group. In this paper we shall assume that $G$ is split, so it has a maximal torus $T\simeq (\GL_1)^r$ over $\QQ$. Let $X^*(T)=\Hom(T,\GL_1)$ and $X_*(T)=\Hom(\GL_1,T)$ be the character and cocharacter groups, respectively, of $T$. There is a natural pairing $\langle,\rangle:X^*(T)\times X_*(T)\rightarrow \ZZ$. Let $W_G=W=N_G(T)/T$ be the Weyl group. Let $\Phi\subset X^*(T)$ be the set of roots, $\Phi^+=\Phi^+_G$ the set of positive roots (with respect to a fixed ordering), and $\Delta_G$ the set of simple positive roots. Let $\rho_G$ be half the sum of all the positive roots. Given any root $\alpha$, there is an associated coroot $\check{\alpha}\in X_*(T)$, with $\langle \alpha,\check{\alpha}\rangle=2$. If $\langle,\rangle'$ is any $W$-invariant inner product on $X^*(T/S)\otimes\RR$, where $S=Z(G)^0$ is the connected component of the identity in the centre of $G$, then for any root $\alpha$, and any $\chi\in X^*(T/S)$, we have $\langle\chi,\check{\alpha}\rangle=\langle\chi,\frac{2\alpha}{\langle\alpha,\alpha\rangle'}\rangle'$. Identifying $\check{\alpha}$ with $\frac{2\alpha}{\langle\alpha,\alpha\rangle'}$ we get an isomorphism $X_*(T/S)\otimes\RR\simeq X^*(T/S)\otimes\RR$, so from now on we write $\langle,\rangle'$ as $\langle,\rangle$. Let $B$ be the Borel subgroup (minimal parabolic) of $G$ corresponding to $\Phi^+_G$.

If we choose any $\alpha\in\Delta_G$ then there is a maximal parabolic subgroup $P=MN$ of $G$, with unipotent radical $N$ and (reductive) Levi subgroup $M$, characterised by $\Delta_M=\Delta_G-\{\alpha\}$. The roots in $\Phi$ are those non-trivial characters of $T$ arising from its adjoint action on the Lie algebra $\g$ of the algebraic group $G$. Let $\Phi_N$ be the subset occurring in the Lie algebra $\n$ of $N$, i.e. those elements of $\Phi^+_G$ whose decomposition as a sum of simple roots includes $\alpha$, and let $\rho_P$ be half the sum of the elements of $\Phi_N$. Let $\tilde{\alpha}:=\frac{\rho_P}{\langle\rho_P,\check{\alpha}\rangle}$. Then $\langle\tilde{\alpha},\check{\alpha}\rangle=1$, while $\langle\tilde{\alpha},\check{\beta}\rangle=0$ for all other simple positive roots $\beta$ (as can be seen by considering the action of $W_M$), i.e. $\tilde{\alpha}$ is a fundamental dominant weight in $X^*(T/S)$.

Let $\hat{G}$ be the Langlands dual group of $G$ \cite[Chapter 3]{Ki}, \cite[I.2]{Bo}. Then $\hat{G}$ has a maximal torus $\hat{T}$ with $X^*(\hat{T})\simeq X_*(T)$ and $X_*(\hat{T})\simeq X^*(T)$. Under these isomorphisms, roots of $\hat{G}$ become coroots of $G$, and coroots of $\hat{G}$ become roots of $G$, with $\check{\Delta}:=\{\check{\beta}:\beta\in\Delta_G\}$ mapping to a set of simple positive roots for $\hat{G}$. We can define a maximal parabolic subgroup $\hat{P}$ of $\hat{G}$, with Levi subgroup characterised by having set of simple positive roots $\check{\Delta}-\{\check{\alpha}\}$, hence identifiable with $\hat{M}$. Let $\hat{N}$ be the unipotent radical of $\hat{P}$, with Lie algebra $\hat{\n}$.

Letting $A:=Z(M)^0$, the restriction map from $X^*(M)$ (i.e. $\Hom(M,\GL_1)$) to $X^*(A)$ identifies $X^*(M)$ with a finite-index subgroup of $X^*(A)$, thus $X^*(M)\otimes\RR=X^*(A)\otimes\RR$. If $\chi\in X^*(M)$ then we can define, for any archimedean or non-archimedean place $v$ of $\Q$, a homomorphism $|\chi|_v:M(\Q_v)\rightarrow \RR^{\times}$ by $|\chi|_v(m)=|\chi(m)|_v$. We can extend this to $X^*(M)\otimes\RR$, or even $X^*(M)\otimes\CC$, by $|s\chi|_v(m)=|\chi(m)|_v^s$ in $\CC^{\times}$. For a finite prime $p$, $|\cdot|_p$ is normalised so that $|p|_p=p^{-1}$. Let $\A$ be the adele ring of $\Q$, and let $G(\A)$ be the group of points of the $\QQ$-algebraic group $G$ in the $\QQ$-algebra $\A$. Taking a product over all places, we may define, for any $\chi\in X^*(M)\otimes\RR$, a homomorphism $|\chi|:M(\A)\rightarrow \RR^{\times}$. In particular, restricting $2\rho_P$ to $A$ then viewing it in $X^*(A)\otimes\RR=X^*(M)\otimes\RR$, we have $|s\tilde{\alpha}|:M(\A)\rightarrow\CC^{\times}$, for any $s\in\CC$. Note that this character is trivial when restricted to $S(\A)$.

Let $\Pi$ be an irreducible, cuspidal, automorphic representation of $M(\A)$. We shall assume in addition that $\Pi$ is unitary, and that it is trivial on $A(\A)$. (This latter assumption is purely for simplicity. Without it, we could, for instance, in the case $M=\GL_1\times\GL_1$ in \S 5 below, let $\Pi=\psi_1\times\psi_2$ for two Dirichlet characters, and $\zeta(s)$ would be replaced by the Dirichlet $L$-function $L(\psi_1\psi_2^{-1},s)$.) We have $\Pi=\otimes_v \Pi_v$, where each $\Pi_v$ is an irreducible, admissible representation of $M(\Q_v)$, unramified for all but finitely many $v$. Then $\Pi\otimes|s\tilde{\alpha}|$ is a representation of $M(\A)$, trivial on $S(\A)$. We may parabolically induce it to a representation $\Ind_P^G(\Pi\otimes|s\tilde{\alpha}|)$ of $G(\A)$, trivial on $S(\A)$. This induction is as described in \cite[Chapter 4]{Ki}. It involves the addition of $\rho_P$ to $s\tilde{\alpha}$, with the consequence that $\Ind_P^G(\Pi\otimes|s\tilde{\alpha}|)$ would be unitary if $s\in i\RR$, though we shall always take $s\in\RR_{>0}$.

The admissibility of $\Pi_{\infty}$ follows from it being unitary and irreducible, by a theorem of Harish-Chandra \cite[Theorem 2.3]{Da}. Then, by \cite[Proposition 5.19]{Kn} or \cite[\S 3]{Da}, the centre $Z(\mm_{\CC})$ of the universal enveloping algebra $U(\mm_{\CC})$ (where $\mm_{\CC}$ is the complexification of the Lie algebra of $M(\RR)$) acts by a character (the ``infinitesimal character''), on the dense subspace of $K_{\infty}$-finite vectors, where $K_{\infty}$ is a maximal compact subgroup of $M(\R)$. Given any Cartan subalgebra $\hh_{\CC}$ of $\mm_{\CC}$, the Harish-Chandra isomorphism from $Z(\mm_{\CC})$ to $U(\hh_{\CC})^{W(\hh_{\CC})}$ (invariants under the Weyl group) allows us to write the infinitesimal character in the form $\chi_{\lambda}$ for some $\lambda\in\hh_{\CC}^*$, determined only up to the Weyl group action. See \cite[Theorem 5.62]{Kn} or \cite[\S 3]{Da}. In discussions of discrete series representations, ``compact'' Cartan subalgebras are most directly relevant, but all Cartan subalgebras of $\mm_{\CC}$ are conjugate by $M(\CC)$ \cite[Theorem 2.15]{Kn}, and it is convenient to take $\hh$ to be the Lie algebra of $T(\RR)$ (and $\hh_{\CC}=\hh\otimes_{\RR}\CC$), so we may identify $\lambda$ (and by abuse of notation $\chi_{\lambda}$) with an element of $X^*(T)\otimes\CC$ (which for us will always be in $X^*(T)\otimes\RR$). If $\Pi_{\infty}$ has infinitesimal character $\lambda$ (up to the action of $W_M$), then $\Ind_P^G(\Pi_{\infty}\otimes|s\tilde{\alpha}|_{\infty})$ (i.e. the $\RR$-component of $\Ind_P^G(\Pi\otimes|s\tilde{\alpha}|)$), though not in general unitary, has an infinitesimal character $\lambda+s\tilde{\alpha}$, now only determined up to the action of $W_G$. Applying an element of $W_M$ if necessary, we may arrange for $\lambda$ to be dominant with respect to $\Delta_M$, i.e. $\langle \lambda,\check{\beta}\rangle\geq 0$ for all $\beta\in\Delta_M$. This follows from \cite[Theorem 2.63, Proposition 2.67]{Kn}. However, $\lambda$ might not be dominant for $\Delta_G$: it might not be the case that $\langle \lambda,\check{\alpha}\rangle\geq 0$, but similarly there exists some $w\in W_G$ such that $w(\lambda)$ is dominant for $\Delta_G$. Note that the finite-dimensional representation of $G$ with highest weight $\lambda$ has infinitesimal character $\lambda+\rho_G$.

\begin{lem}\label{unitary} Let $\lambda\in X^*(T)\otimes\RR$ be the infinitesimal character of a unitary irreducible representation of $M(\RR)$, and suppose that $\lambda$ is chosen in its $W_M$-orbit to be dominant. Then $\lambda=-w_0^M\lambda$, where $w_0^M\in W_M$ is the long element, exchanging positive and negative roots.
\end{lem}
\begin{proof} Since the representation is unitary, its conjugate (i.e. $V\otimes_{\sigma}\CC$, where $\sigma$ is complex conjugation, so all matrix coefficients are conjugated) and its dual are isomorphic. Since $M$ is split, the infinitesimal character of the conjugate is also $\lambda$. The infinitesimal character of the dual (chosen dominant in its $W_M$-orbit) is $-w_0^M\lambda$.
\end{proof}

Let $p$ be a finite prime such that $\Pi_p$ is unramified (or ``spherical''), i.e. has a non-zero $M(\ZZ_p)$-fixed (``spherical'') vector. Note that $M(\ZZ_p)$ is defined using the Chevalley group scheme for the split group $M$, and likewise for $G(\ZZ_p)$. Then for some $\chi_p\in X^*(T)\otimes i\RR$, $\Pi_p$ is isomorphic to a unique irreducible quotient of the (unitarily) parabolically induced representation $\Ind_{B_M(\QQ_p)}^{M(\QQ_p)}(|\chi_p|_p)$ \cite[Theorem 4.17]{Ki}, \cite[4.4(d)]{Ca}, where $B_M:=B\cap M$. Note that $\chi_p$ can be replaced by anything in the same $W_M$-orbit, and that
the character $|\chi_p|_p$ of $T(\QQ_p)$ is unramified, i.e. trivial on $T(\ZZ_p)$. Also, $\Ind_{B_M(\QQ_p)}^{M(\QQ_p)}(|\chi_p|_p)$  is irreducible if $\chi_p$ is regular (i.e. if $\langle\chi_p,\check{\beta}\rangle\neq 0$, for every $\beta\in\Phi_M$). The local component at $p$ of  $\Ind_P^G(\Pi\otimes|s\tilde{\alpha}|)$ is easily seen, from the definition of induction \cite[Ch.4,\S 2]{Ki}, to have a $G(\ZZ_p)$-fixed vector, and by transitivity of induction \cite[Lemma 6.1]{Ki},\cite[I(36)]{Ca} it is a subquotient of $\Ind_{B(\QQ_p)}^{G(\QQ_p)}(|\chi_p+s\tilde{\alpha}|_p)$. Hence it has the spherical subquotient of $\Ind_{B(\QQ_p)}^{G(\QQ_p)}(|\chi_p+s\tilde{\alpha}|_p)$ as an irreducible constituent. Note that $|\chi_p+s\tilde{\alpha}|_p$ is still an unramified character of $T(\QQ_p)$, though it is not unitary for $s\notin i\RR$. In our application, $\chi_p$ will always be regular for $M$, and $s$ chosen so that $\chi_p+s\tilde{\alpha}$ is regular for $G$, hence $\Ind_{P(\QQ_p)}^{G(\QQ_p)}(\Pi_p\otimes|s\tilde{\alpha}|_p)$ will be irreducible.

We refer to $\chi_p$ and $\chi_p+s\tilde{\alpha}$ as the Satake parameters at $p$ of $\Pi$ and $\Ind_P^G(\Pi\otimes|s\tilde{\alpha}|)$, respectively. Let $\h=\h(G(\QQ_p),G(\ZZ_p))$ be the Hecke algebra of $\CC$-valued, compactly supported, $G(\ZZ_p)$-bi-invariant functions on $G(\QQ_p)$. If $f\in\h$ then $f$ acts on $\Ind_{P(\QQ_p)}^{G(\QQ_p)}(\Pi_p\otimes |s\tilde{\alpha}|_p)$ (or any other representation of $G(\QQ_p)$) by $v\mapsto \int_{G(\QQ_p)}g(v)f(g)\,dg$, where $dg$ is a left- and right-invariant Haar measure, normalised so that $G(\ZZ_p)$ has volume $1$. Then $\h$ is a commutative ring under convolution of functions (which corresponds to composition of operators), and is generated by the characteristic functions $T'_{\mu}$ of double cosets $G(\ZZ_p)\mu(p) G(\ZZ_p)$, where $\mu\in X_*(T)$ is any cocharacter. If $v_0$ is a spherical vector then necessarily so is $T'_{\mu}(v_0)$, but since $v_0$ is unique up to scalar multiples, $\h$ acts on $v_0$ by a character. The value of this character on any particular element of $\h$ is a ``Hecke eigenvalue''. (When a classical cuspidal Hecke eigenform is identified with a vector in an automorphic representation of $\GL_2(\A)$, this vector is spherical locally at primes not dividing the level.)

Given $\chi\in X^*(T)\otimes\CC$, there is $t(\chi)\in \hat{T}(\CC)$ such that, for any $\mu\in X_*(T)=X^*(\hat{T})$, $\mu(t(\chi))=|\chi(\mu(p))|_p$.
In the case $\chi=s\lambda$, with $\lambda\in X^*(T)=X_*(\hat{T})$ and $s\in\CC$, we have $t(\chi)=\lambda(p^{-s})$, and $\mu(t(\chi))=|\chi(\mu(p))|_p=p^{-s\langle\lambda,\mu\rangle}$. The Hecke eigenvalue for $T'_{\mu}$, on the spherical representation of $G(\QQ_p)$ with Satake parameter $\chi$ (or $t(\chi)$, thought of as a conjugacy class in $\hat{G}(\CC)$) may be calculated using the Satake isomorphism. In particular, if $\mu$ is minuscule, meaning that the orbit of $\mu$ under $W_G$ is the set of weights for the irreducible representation $\theta_{\mu}$ of $\hat{G}$ with highest weight $\mu$, then the eigenvalue is $p^{\langle \rho_G,\mu\rangle}\Tr(\theta_{\mu}(t(\chi)))=p^{\langle \rho_G,\mu\rangle}\sum_{w\in W_G}|\chi(w(\mu)(p))|_p$ \cite[3.13,6.2]{Gr}. Similarly for spherical representations of $M(\QQ_p)$.

\section{Motives and $L$-functions}
Recall that the representation $\Pi$ of $M(\A)$, at an unramified prime $p$, has a Satake parameter $\chi_p\in X^*(T)\otimes i\RR$, or $t(\chi_p)\in \hat{T}(\CC)\subset \hat{M}(\CC)$. Given a representation $r:\hat{M}\rightarrow \GL_d$, we may define a local $L$-factor
$$L_p(s,\Pi_p,r):=\det(I-r(t(\chi_p))p^{-s})^{-1},$$
and an $L$-function (in general incomplete)
$$L_{\Sigma}(s,\Pi,r):=\prod_{p\notin \Sigma}L_p(s,\Pi_p,r),$$
where $\Sigma$ is a finite set of primes containing all those such that $\Pi_p$ is ramified (i.e. not spherical).

In particular, we take for $r$ the adjoint representation of $\hat{M}$ on the Lie algebra $\hat{\n}$ of the unipotent radical of the maximal parabolic $\hat{P}$. Now $\hat{\n}$ is a direct sum of subspaces on which $\hat{T}$ acts by those positive roots of $\hat{G}$ that are not roots of $\hat{M}$. These are identified with the coroots $\check{\gamma}$ of $G$, as $\gamma$ runs through $\Phi_N$. It follows that
$$L_p(s,\Pi_p,r)^{-1}=\prod_{\gamma\in\Phi_N}(1-\check{\gamma}(t(\chi_p))p^{-s})=\prod_{\gamma\in\Phi_N}(1-|\chi_p(\check{\gamma}(p))|_pp^{-s}).$$
Actually, $r$ is a direct sum of irreducible representations $r_i$ for some $1\leq i\leq m$,
where $r_i$ acts on the direct sum $\hat{\n}_i$ of root spaces for $\{\check{\gamma}:\,\gamma\in\Phi_N^i\}$ with
$$\Phi_N^i:=\{\gamma\in\Phi_N:\langle\tilde{\alpha},\check{\gamma}\rangle=i\},$$
see \cite[Theorem 6.6]{Ki}, and so
$$L_{\Sigma}(s,\Pi,r)=\prod_{i=1}^mL_{\Sigma}(s,\Pi,r_i).$$ Note that $L_{\Sigma}(0,\Pi\otimes|s\tilde{\alpha}|,r_i)=L_{\Sigma}(is,\Pi,r_i)$, and beware that here $i$ is not $\sqrt{-1}$.

Let $s\in\RR$ be chosen so that $\lambda+s\tilde{\alpha}\in X^*(T)$. Then according to \cite[Conjecture 3.2.1]{BG}, there should exist a continuous representation $\rho_{\Pi\otimes |s\tilde{\alpha}|}:\Gal(\Qbar/\QQ)\rightarrow \hat{M}(E_{\q})$, such that if $p\notin \Sigma\cup\{q\}$ then
$\rho_{\Pi\otimes |s\tilde{\alpha}|}$ is unramified at $p$, with $\rho_{\Pi\otimes |s\tilde{\alpha}|}(\Frob_p^{-1})$ conjugate in $\hat{M}(E_{\q})$ to $t(\chi_p+s\tilde{\alpha})$. Here, $E$ is a certain field of definition of the Satake parameters
\cite[Definitions 2.2.1, 3.1.3]{BG}, and $\q$ is any prime divisor. Moreover, by \cite[Conjecture 4.5]{Cl} (applied to the conjectural functorial lift of $\Pi\otimes |s\tilde{\alpha}|$ to $\GL_d(\A)$), $r\circ\rho_{\Pi\otimes |s\tilde{\alpha}|}$ should be the $\q$-adic realisation of a motive $\MMM(r,\Pi\otimes |s\tilde{\alpha}|)$. In fact, this should be a direct sum $\oplus_{i=1}^m\MMM(r_i,\Pi\otimes |s\tilde{\alpha}|)$.
We shall assume the existence of these motives, or at least of the associated premotivic structures (realisations and comparison isomorphisms).

In fact, we need to make a weaker assumption than that $\lambda+s\tilde{\alpha}\in X^*(T)$. We assume only that $\lambda+s\tilde{\alpha}$ is {\em algebraically integral}, i.e. that $\langle\lambda+s\tilde{\alpha},\check{\gamma}\rangle\in\ZZ$ for all $\gamma\in\Phi$. It will no longer necessarily be the case that the Satake parameters of $\Pi\otimes |s\tilde{\alpha}|$ are defined over a number field, but those of the lift to $\GL_d(\A)$ should be, and we take $E$ to be their field of definition.

If $r_{\infty}:W_{\RR}\rightarrow\hat{M}(\CC)$ is the Langlands parameter at $\infty$ (of $\Pi\otimes|s\tilde{\alpha}|$) then, restricting to the subgroup $\CC^{\times}$, of index two in the Weil group $W_{\RR}$, $r_{\infty}(z)$ is conjugate in $\hat{M}(\CC)$ to $(\lambda+s\tilde{\alpha})(z)(\lambda'+s\tilde{\alpha})(\overline{z})$, where $\lambda'$ is in the same $W_M$-orbit as $\lambda$ \cite[\S 2.3]{BG}. (Note that $\tilde{\alpha}$ is fixed by $W_M$.) Actually, because $\Pi_{\infty}$ is unitary, we must have $\lambda'=-\lambda$ (which is in the $W_M$-orbit of $\lambda$ also by Lemma \ref{unitary}). Then $r_i\circ r_{\infty}(z)$ is conjugate in $\GL_{d_i}(\CC)$ to $\diag(z^{\langle\lambda+s\tilde{\alpha},\check{\gamma}\rangle}\overline{z}^{\langle-\lambda+s\tilde{\alpha},\check{\gamma}\rangle})_{\gamma\in\Phi_N^i}=
\diag(z^{\langle\lambda,\check{\gamma}\rangle +is}\,\overline{z}^{-\langle\lambda,\check{\gamma}\rangle +is})_{\gamma\in\Phi_N^i}$.
It follows that the Hodge type of $\MMM(r_i,\Pi\otimes |s\tilde{\alpha}|)$ should be $\{(-\langle\lambda,\check{\gamma}\rangle -is,\langle\lambda,\check{\gamma}\rangle - is): \gamma\in\Phi_N^i\}$. (For the minus sign, see \cite[1.1.1.1]{De2}. This accords with the fact that making a positive Tate twist reduces the weight.) The complex conjugation $F_{\infty}$ on the Betti realisation $H_B(\MMM(r_i,\Pi\otimes |s\tilde{\alpha}|))\otimes\CC$ should exchange $(p,q)$ and $(q,p)$, so the next lemma is no surprise.

\begin{lem}\label{symmetry} If $\gamma\in \Phi_N$ then $\gamma':=w_0^M\gamma$ is also in $\Phi_N$, and $\langle\lambda,\check{\gamma'}\rangle=-\langle\lambda,\check{\gamma}\rangle$.
\end{lem}
\begin{proof}
We have that $w_0^M$ is represented by the conjugation action of some element of $M$, which preserves $N$, so $\gamma'\in\Phi_N$. By Lemma \ref{unitary}, $\lambda=-w_0^M\lambda$ and hence $\langle\lambda,\gamma'\rangle=-\langle w_0^M\lambda,w_0^M\gamma\rangle=-\langle\lambda,\gamma\rangle$. Now multiply by $\frac{2}{\langle\gamma',\gamma'\rangle}=\frac{2}{\langle\gamma,\gamma\rangle}$.
\end{proof}
In fact, it is easy to see that if $\gamma\in\Phi_N^i$ then $\gamma'\in\Phi_N^i$.

We shall be especially concerned with the Tate twist $\MMM(r_i,\Pi\otimes |s\tilde{\alpha}|)(1)$. Let $H_B(\MMM(r_i,\Pi\otimes |s\tilde{\alpha}|)(1))$ and $H_{\dR}(\MMM(r_i,\Pi\otimes |s\tilde{\alpha}|)(1))$ be the Betti and de Rham realisations, and let $H_B(\MMM(r_i,\Pi\otimes |s\tilde{\alpha}|)(1))^{\pm}$ be the eigenspaces for the complex conjugation $F_{\infty}$. As in \cite[1.7]{De1},
$\MMM(r_i,\Pi\otimes |s\tilde{\alpha}|)(1)$ is {\em critical} if $\dim(H_B(\MMM(r_i,\Pi\otimes |s\tilde{\alpha}|)(1))^+)=\dim(H_{\dR}(\MMM(r_i,\Pi\otimes |s\tilde{\alpha}|)(1))/\Fil^0)$. Let $\wt=-2is-2$ be the weight of $\MMM(r_i,\Pi\otimes |s\tilde{\alpha}|)(1)$ (so $p+q=\wt$ for all $(p,q)$ in the Hodge type), and let $h^{p,q}$ be the dimension of the $(p,q)$-part $H^{p,q}$ of $H_B(\MMM(r_i,\Pi\otimes |s\tilde{\alpha}|)(1))\otimes\CC$. Note that $F_{\infty}$ exchanges $H^{p,q}$ and $H^{q,p}$, so $h^{p,q}=h^{q,p}$.

\begin{prop}\label{critical}
Let $b_i$ be the smallest non-zero positive value of $\langle\lambda,\check{\gamma}\rangle$, for $\gamma\in\Phi_N^i$.
\begin{enumerate}
\item If $\wt$ is odd, or if $\wt$ is even but $h^{\wt/2,\wt/2}=0$, then $\MMM(r_i,\Pi\otimes |s\tilde{\alpha}|)(1)$ is critical for $0<s\leq\frac{b_i-1}{i}$ (subject also to $\lambda+s\tilde{\alpha}$ being algebraically integral).
\item If $\wt$ is even and $h^{\wt/2,\wt/2}\neq 0$, suppose that $F_{\infty}$ acts on $H^{\wt/2,\wt/2}$ by a scalar (necessarily $(-1)^t$ with $t=0$ or $1$). Then $\MMM(r_i,\Pi\otimes |s\tilde{\alpha}|)(1)$ is critical for $0<s\leq\frac{b_i-1}{i}$, subject also to $\lambda+s\tilde{\alpha}$ being algebraically integral, and the extra condition $t=0$. (Note that whenever $s$ goes up by $1$, $\MMM(r_i,\Pi\otimes |s\tilde{\alpha}|)(1)$ gets Tate twisted by $i$, so $t$ changes by $i\pmod{2}$. So the condition $t=0$ amounts to a kind of parity condition on $is$.)
\end{enumerate}
\end{prop}
\begin{proof}
\begin{enumerate}
\item In this case, $$\dim(H_B(\MMM(r_i,\Pi\otimes |s\tilde{\alpha}|)(1))^+)=\frac{1}{2}\dim(H_B(\MMM(r_i,\Pi\otimes |s\tilde{\alpha}|)(1))).$$ We have also $$\dim(H_{\dR}(\MMM(r_i,\Pi\otimes |s\tilde{\alpha}|)(1))/\Fil^0)=\frac{1}{2}\dim(H_{\dR}(\MMM(r_i,\Pi\otimes |s\tilde{\alpha}|)(1)))$$ if and only if, when the $(p,q)$, with multiplicities, are listed in order of increasing $p$, the $p$ immediately to the right of the centre is non-negative. This is to say that $b_i-is-1\geq 0$, i.e. that $s\leq \frac{b_i-1}{i}$.
\item This follows similarly from $$\dim(H_{\dR}(\MMM(r_i,\Pi\otimes |s\tilde{\alpha}|)(1))/\Fil^0)=\frac{1}{2}[\dim(H_{\dR}(\MMM(r_i,\Pi\otimes |s\tilde{\alpha}|)(1)))+h^{\wt/2,\wt/2}].$$
\end{enumerate}
\end{proof}
Note that in case (2), if $F_{\infty}$ did not act on $H^{\wt/2,\wt/2}$ by a scalar then there would be no critical values. The proposition describes all the positive $s$ for which $\MMM(r_i,\Pi\otimes |s\tilde{\alpha}|)(1)$ is critical. We ignore negative $s$. In terms of $L$-functions, we are ignoring critical values that are central or left-of-centre. (In fact, if $\wt$ is even we are also ignoring any critical value immediately to the right of centre, by excluding $s=0$.) If there is no non-zero value of $\langle\lambda,\check{\gamma}\rangle$ for $\gamma\in\Phi_N^i$ then there is no upper bound on $s$--we might say that $b_i=\infty$.

We have that $\MMM(r,\Pi\otimes |s\tilde{\alpha}|)(1)$ is critical for $0<s\leq\min_{i}\frac{b_i-1}{i}$, subject also to $\lambda+s\tilde{\alpha}$ being algebraically integral, and the simultaneous parity conditions. Recall that we chose $w\in W_G$ such that $w(\lambda)$ is dominant. If $\lambda$ is on the wall of a Weyl chamber then there is more than one possible choice of $w$. Assuming we have chosen $\lambda$ to be strictly dominant for $M$, i.e. $\langle\lambda,\check{\beta}\rangle>0$ for all $\beta\in\Delta_M$ (hence for all $\beta\in\Phi_M^+$), $\lambda$ can only be on the wall of a Weyl chamber if $\langle\lambda,\check{\gamma}\rangle=0$ for some $\gamma\in\Phi_N$. Since $\langle\tilde{\alpha},\check{\gamma}\rangle=i\neq 0$ for $\gamma\in\Phi_N^i$, for $s>0$ sufficiently small $\lambda+s\tilde{\alpha}$ is not on the wall of a Weyl chamber. Hence we may (and shall) choose $w$ so that $w(\lambda+s\tilde{\alpha})$ is strictly dominant, for $s>0$ sufficiently small, i.e. $\langle w(\lambda+s\tilde{\alpha}),\check{\beta}\rangle>0$ for all $\beta\in\Delta_G$. In other words, $w(\lambda+s\tilde{\alpha})$ is dominant and regular. Once $s$ reaches $\min_{i}\frac{b_i}{i}$, which is when $\MMM(r,\Pi\otimes |s\tilde{\alpha}|)(1)$ stops being critical, also some $\langle\lambda+s\tilde{\alpha},\check{\gamma}\rangle$ reaches $0$ (with $\gamma\in\Phi_N^i$ such that $\langle \lambda,\check{\gamma}\rangle=-b_i$), and is about to change sign. Thus $\lambda+s\tilde{\alpha}$ has reached the wall of a Weyl chamber, as has $w(\lambda+s\tilde{\alpha})$, so $w(\lambda+s\tilde{\alpha})$ ceases to be strictly dominant at this point. We have then a remarkable correspondence between critical arguments and strictly dominant weights, which will be illustrated in the examples later.

\section{The main conjecture}
Recall the field of definition $E$ from the previous section. Suppose that $s>0$ and that $\MMM(r,\Pi\otimes |s\tilde{\alpha}|)(1)$ is critical. Let $\q$ be a prime divisor, dividing a rational prime $q$ such that $\Pi_q$ is unramified and such that $q>\mathcal{B}_i$, where
$$\mathcal{B}_i:=\begin{cases} 2\max_{\gamma\in\Phi_N^i}\langle\lambda,\check{\gamma}\rangle+1 & \text{ if $\max_{\gamma\in\Phi_N^i}\langle\lambda,\check{\gamma}\rangle\neq 0$;}\\ 2+is & \text{ if $\max_{\gamma\in\Phi_N^i}\langle\lambda,\check{\gamma}\rangle=0$}.\end{cases}$$
Let $O_{\q}$ be the ring of integers of the completion $E_{\q}$, and $O_{(\q)}$ the localisation at $\q$ of the ring of integers $O_E$ of $E$. For $1\leq i\leq m$, choose an $O_{(\q)}$-lattice $T_{i,B}$ in $H_B(\MMM(r_i,\Pi\otimes |s\tilde{\alpha}|))$ in such a way that $T_{i,\q}:=T_{i,B}\otimes O_{\q}$ is a $\Gal(\Qbar/\Q)$-invariant lattice in the $\q$-adic realisation. Then choose an $O_{(\q)}$-lattice $T_{i,\dR}$ in
 $H_{\dR}(\MMM(r_i,\Pi\otimes |s\tilde{\alpha}|))$ in such a way that
$$\VVV(T_{i,\dR}\otimes O_{\q})=T_{i,\q}$$ as $\Gal(\Qbar_q/\QQ_q)$-representations, where $\VVV$ is the version of the Fontaine-Lafaille functor used in \cite{DFG}. Since $\VVV$ only applies to filtered $\phi$-modules, where $\phi$ is the crystalline Frobenius, $T_{i,\dR}$ must be $\phi$-stable. Anyway, this choice ensures that the $\q$-part of the Tamagawa factor at $q$ is trivial (by \cite[Theorem 4.1(iii)]{BK}), thus simplifying the Bloch-Kato conjecture below. The condition $q>\mathcal{B}_i$ ensures that the condition (*) in \cite[Theorem 4.1(iii)]{BK} holds.

Let $\Omega$ be a Deligne period scaled according to the above choice, i.e. the determinant of the isomorphism
$$H_B(\MMM(r_i,\Pi\otimes |s\tilde{\alpha}|)(1))^+\otimes\CC\simeq (H_{\dR}(\MMM(r_i,\Pi\otimes |s\tilde{\alpha}|)(1))/\Fil^0)\otimes\CC,$$
calculated with respect to bases of $T_{i,B}^+$ and $T_{i,\dR}/\Fil^1$, so well-defined up to $O_{(\q)}^{\times}$. As before, let $\Sigma$ be a finite set of finite primes, containing all $p$ such that $\Pi_p$ is ramified, but it should now not contain $q$.

In Case (1) below, the formulation of the Bloch-Kato conjecture is based on \cite[(59)]{DFG}, using the exact sequence in their Lemma 2.1.
\begin{conj}[Bloch-Kato]\label{BK}  $ $
\begin{enumerate}
\item If $\Sigma\neq \emptyset$ then $$\ord_{\q}\left(\frac{L_{\Sigma}(1+is,\Pi,r_i)}{\Omega}\right)=\ord_{\q}\left(\frac{\#H^1_{\Sigma}(\Q,T_{i,\q}^*\otimes (E_{\q}/O_{\q}))}{\#H^0(\Q,T_{i,\q}^*\otimes(E_{\q}/O_{\q}))}\right).$$
\item If $\Sigma=\emptyset$ then $$\ord_{\q}\left(\frac{L(1+is,\Pi,r_i)}{\Omega}\right)$$
$$=\ord_{\q}\left(\frac{\#H^1_{\emptyset}(\Q,T_{i,\q}^*\otimes (E_{\q}/O_{\q}))}{\#H^0(\Q,T_{i,\q}^*\otimes(E_{\q}/O_{\q}))\#H^0(\Q,T_{i,\q}(1)\otimes(E_{\q}/O_{\q}))}\right).$$
\end{enumerate}
\end{conj}
Here, $T_{i,\q}^*=\Hom_{O_{\q}}(T_{i,\q},O_{\q})$, with the dual action of $\Gal(\Qbar/\Q)$, and $\#$ denotes a Fitting ideal. On the right hand side, in the numerator is a Bloch-Kato Selmer group with local conditions (unramified at $p\neq q$, crystalline at $p=q$) only at $p\notin \Sigma$. The ratio of the two sides is independent of the choice of $\Sigma$ as above.

Let $\tilde{\Pi}$ denote any irreducible, cuspidal, tempered, automorphic representation of $G(\A)$ such that $\tilde{\Pi}_{\infty}$ has infinitesimal character $\lambda+s\tilde{\alpha}$ (up to $W_G$), such as that appearing in Conjecture \ref{main} below. 

Recall $T'_{\mu}\in\h=\h(G(\QQ_p),G(\ZZ_p))$, the characteristic function of the double coset $G(\ZZ_p)\mu(p) G(\ZZ_p)$, where $\mu\in X_*(T)$ is any cocharacter. Let $a(\mu):=\langle w(\lambda+s\tilde{\alpha})-\rho_G,\mu\rangle$, and let $T_{\mu}:=p^{a(\mu)}T'_{\mu}$. In the case that $\mu$ is minuscule, the Hecke eigenvalue for $T_{\mu}$ on the spherical representation $\Ind_P^G(\Pi_p\otimes|s\tilde{\alpha}|_p)$,
or on $\tilde{\Pi}_p$,  is $p^{\langle w(\lambda+s\tilde{\alpha}),\mu\rangle}\Tr(\theta_{\mu}(t(\chi)))$, where $\chi$, or $t(\chi)$, is the Satake parameter. Recall the end of \S 2, and that $\theta_{\mu}$  is the irreducible representation of $\hat{G}$ with highest weight $\mu$. This $\Tr(\theta_{\mu}(t(\chi)))$ is the trace of $\Frob_p^{-1}$ for the motive conjecturally associated to $\Pi\otimes|s\tilde{\alpha}|$ (or $\tilde{\Pi}$) and the representation $\theta_{\mu}$ of $\hat{G}$ (which is restricted to $\hat{M}$ in the former case). Multiplying by the power of $p$ corresponds to taking a big enough Tate twist to make all the Hodge numbers non-negative, as they would be for the cohomology of a nonsingular projective variety. Therefore we expect the Hecke eigenvalues for the $T_{\mu}$ to be algebraic integers. For a different way of arriving at the same power of $p$, see \cite[2.3.1(25)]{H6}.

In what follows, we enlarge the field $E$ to be a common field of definition for the Hecke eigenvalues of $T_{\mu}$ (for all $\mu\in X_*(T)$) on the $\Ind_P^G(\Pi_p\otimes|s\tilde{\alpha}|_p)$ and the $\tilde{\Pi}_p$ (for all unramified $p$), and replace $\q$ by any divisor in this possibly larger field.
Let $T_{\mu}(\Ind_P^G(\Pi_p\otimes|s\tilde{\alpha}|_p))$ and $T_{\mu}(\tilde{\Pi}_p)$ denote the Hecke eigenvalues.

\begin{conj}\label{main} Choose $s>1$ such that $\MMM(r,\Pi\otimes |s\tilde{\alpha}|)(1)$ is critical. Suppose that $\lambda+s\tilde{\alpha}$ is self-dual, i.e. $W_G$-equivalent to its negative. Now fixing $i$, with $\q$ and $\Sigma$ as above (in particular, $q>\mathcal{B}_i$), suppose that
$$\ord_{\q}\left(\frac{L_{\Sigma}(1+is,\Pi,r_i)}{\Omega}\right)>0.$$
Suppose also that the irreducible components of the $\q$-adic representation $r_i\circ\rho_{\Pi\otimes |s\tilde{\alpha}|}$ remain irreducible mod $\q$. Then there exists an irreducible, cuspidal, tempered, automorphic representation $\tilde{\Pi}$ of $G(\A)$ such that
\begin{enumerate}
\item $\tilde{\Pi}_{\infty}$ has infinitesimal character $\lambda+s\tilde{\alpha}$ (up to $W_G$), i.e. the same as $\Ind_P^G(\Pi\otimes|s\tilde{\alpha}|)$.
\item At any finite $p\notin \Sigma$, $\tilde{\Pi}_p$ is unramified, and for all $\mu\in X_*(T)$,
$$T_{\mu}(\tilde{\Pi}_p)\equiv T_{\mu}(\Ind_P^G(\Pi_p\otimes|s\tilde{\alpha}|_p))\pmod{\q}.$$
\end{enumerate}
\end{conj}

The self-duality condition is necessary so that there is a possibility of $\lambda+s\tilde{\alpha}$ being the infinitesimal character of a unitary representation. In all the examples below, the long element $w_0^G$ of $W_G$ sends any dominant element of $X^*(T/S)$ to its negative, so the condition is automatically satisfied. But there could be problems in other cases, for example when $G=\GL_3$ and $M\simeq \GL_1\times\GL_2$.

Whether or not
$$\ord_{\q}\left(\frac{L_{\Sigma}(1+is,\Pi,r_i)}{\Omega}\right)>0,$$
could depend on the choice of $T_{i,\q}$ (up to scaling), whereas there either is or is not a $\tilde{\Pi}$ satisfying the congruence. Imposing the condition that the irreducible components of $r_i\circ\rho_{\Pi\otimes |s\tilde{\alpha}|}$ remain irreducible mod $\q$ makes the intersection of $T_{i,\q}$ with each irreducible component unique up to scaling, thus resolving the ambiguity. (Note that the irreducible components of the $\q$-adic realisation should correspond to irreducible components of the motive $\MMM(r_i,\Pi\otimes |s\tilde{\alpha}|)$, and we should perhaps look at each one separately.) This is a condition we shall not keep repeating in each case of the conjecture for the remainder of the paper. In one or two examples in this paper, such as the $q=41$ example near the end of \S 9, and the second $q=691$ example following it, the condition is not satisfied but the congruence seems to work anyway. However, there is an example (examined in detail elsewhere) with $G=\SO(4,3)$, $M\simeq\GL_2\times\SO(2,1)$, $q=691$, in which the condition is not satisfied, and the congruence cannot hold because there is no $\tilde{\Pi}$ with the required infinitesimal character. Thus the condition does seem to be necessary in general.

\section{Example: $G=\GL_2$.}
Let $T=\{\diag(t_1,t_2): t_1,t_2 \in \GL_1\}$
be the standard maximal torus, with character group $X^*(T)=\langle e_1,e_2\rangle_{\ZZ}$, where $e_i:\diag(t_1,t_2)\mapsto t_i$, and cocharacter group $X_*(T)=\langle f_1,f_2\rangle_{\ZZ}$, where $f_1:t\mapsto \diag(t,1)$ and $f_2:t\mapsto \diag(1,t)$. With the standard ordering, $\Phi^+=\Delta_G=\{e_1-e_2\}$, and $\rho_G=\frac{1}{2}(e_1-e_2)$. The only possible choice is $\alpha=e_1-e_2$, leading to $P$ being the Borel subgroup of upper triangular matrices, with Levi subgroup $M=T\simeq\GL_1\times\GL_1$. Then $\rho_P=\rho_G$, $\langle \rho_P,\check{\alpha}\rangle=1$ and $\tilde{\alpha}=\frac{1}{2}(e_1-e_2)$. The Weyl group $W$ has a non-identity element swapping $e_1$ and $e_2$, and we take the Weyl-invariant inner product on $X^*(T)\otimes\RR$ such that $\langle e_i,e_j\rangle=\delta_{ij}$, restricted to $X^*(T/S)\otimes\RR$, with $S=\{\diag(t,t):\,t\in\GL_1 \}$.

Since $A=M$, the only choice for $\Pi$ is the trivial representation of $M(\A)$, with $\lambda=0$ and $\chi_p=0$ for all $p$. We can take $\Sigma=\emptyset$. We have $\Phi_N=\Phi^1_N=\{e_1-e_2\}$, $r$ is a $1$-dimensional representation of $\hat{M}\simeq\GL_1\times\GL_1$,
$\MMM(r,\Pi\otimes |s\tilde{\alpha}|)=\QQ(s)$ and $L(s,\Pi,r)=\zeta(s)$ is the Riemann zeta function. We must have $s\in\ZZ$ for $\lambda+s\tilde{\alpha}=\frac{s}{2}(e_1-e_2)$ to be algebraically integral, then the weight $\wt=-2s-2$ of $\MMM(r,\Pi\otimes |s\tilde{\alpha}|)(1)=\QQ(s+1)$ is even, and $F_{\infty}$ acts on $H^{\wt/2,\wt/2}$ by the scalar $(-1)^{1+s}$.
Hence the condition $t=0$ in Proposition \ref{critical} becomes $s$ odd. If $s>1$ then $s+1=k$ with $k\geq 4$ even.

Now $\lambda+s\tilde{\alpha}=\frac{k-1}{2}(e_1-e_2)$, which is already dominant, without having to apply any element of $W_G$. We recognise it as the infinitesimal character of the discrete series representation
$D_k$ of $\GL_2(\R)$, which is $\tilde{\Pi}_{f,\infty}$ for the cuspidal automorphic representation $\tilde{\Pi}_f$ generated by a cuspidal Hecke eigenform $f$ of weight $k$. We have $a(f_1)=\langle \frac{k-1}{2}(e_1-e_2)-\frac{1}{2}(e_1-e_2),f_1\rangle=\frac{k}{2}-1$, and $T_{f_1}=p^{(k/2)-1}T'_{f_1}$ is the standard Hecke operator at $p$. Viewing $f_1\in X^*(\hat{T})$, it is the highest weight of the standard representation of $\hat{G}=\GL_2$, with weights $f_1,f_2$. We see that $f_1$ is a minuscule weight. We have Satake parameter $\chi_p+s\tilde{\alpha}=\frac{k-1}{2}(e_1-e_2)$ for $\Pi_p\otimes|s\tilde{\alpha}|_p$. Then $|\frac{k-1}{2}(e_1-e_2)(f_1(p))|_p=p^{-\frac{k-1}{2}\langle e_1-e_2,f_1\rangle}=p^{-(k-1)/2}$, and similarly $|\frac{k-1}{2}(e_1-e_2)(f_2(p))|_p=p^{(k-1)/2}$. So $\Tr(\theta_{f_1}(t(\chi_p+s\tilde{\alpha})))=p^{(k-1)/2}+p^{-(k-1)/2}$. Multiplying by
$p^{\langle\frac{k-1}{2}(e_1-e_2),f_1\rangle}=p^{(k-1)/2}$, we find that $$T_{f_1}(\Ind_P^G(\Pi_p\otimes|s\tilde{\alpha}|_p))=1+p^{k-1},$$ which we recognise as the Hecke eigenvalue for the holomorphic Eisenstein series of weight $k$, a vector in the space of $\Pi\otimes|s\tilde{\alpha}|$. Note that since $k\geq 4$, the Eisenstein series does converge. In this case, $\mathcal{B}_2=2+2s=k+1$, so the bound on $q$ is $q>k+1$.

The Deligne period for $\QQ(k)$ is $(2\pi i)^k$. If the prime $q>k$ is such that
$$\ord_q\left(\frac{\zeta(k)}{(2\pi i)^k}\right)=\ord_q\left(\frac{B_k}{2k!}\right)>0,$$
where $B_k$ is the Bernoulli number, then Conjecture \ref{main} says that there should be a normalised, cuspidal Hecke eigenform $f=\sum_{n=1}^{\infty}a_n(f)e^{2\pi i n\tau}$ of level $1$, with $E=\Q(\{a_n\})$ and $\q\mid q$ in $O_E$, and it should satisfy $a_p(f)\equiv 1+p^{k-1}\pmod{\q}$ for all primes $p$. (Level $1$ corresponds to $\tilde{\Pi}_f$ being unramified at all finite $p$.) This is the familiar congruence of Ramanujan type (his case being $k=12, q=691$) and here Conjecture \ref{main} is a theorem, see for instance \cite[Theorem 1.1]{DF}.

We can artificially increase the size of $\Sigma$ beyond its minimum, e.g. letting $\Sigma=\{p_0\}$ for some prime $p_0$. This introduces a factor of $(p_0^k-1)/p_0^k$ into $\zeta_{\Sigma}(k)$, so if we allow $\tilde{\Pi}_f$ to be ramified at $p_0$, we should expect to find congruences of the same shape, for all $p\neq p_0$, when $q\mid(p_0^k-1)$. Such congruences, which can be said to be of local origin, specifically for $f\in S_k(\Gamma_0(p_0))$, were predicted by G. Harder (for a different reason \cite[\S 2.9]{H2}), who also gave a numerical example. For a general proof, and further examples, see \cite{DF}.

The case $k=2$ (excluded by $s\neq 1$) is a bit different. Of course, there is no $q$ such that $\ord_q(\zeta(2)/\pi^2)>0$, but even if we enlarge $\Sigma$, say to $\{p_0\}$, then by a theorem of Mazur \cite[Prop. 5.12(ii)]{Ma}, the condition for the congruence (at least with $f\in S_2(\Gamma_0(p_0))$) is $\ord_q((p_0-1)/12)>0$, which does not necessarily hold when $q\mid(p_0^2-1)$. We can include $s=1$ in this case of Conjecture \ref{main} by dropping our insistence that $\tilde{\Pi}_p$ be unramified for all $p\notin\Sigma$, and using $f\in S_2(\Gamma_0(p_1))$ where $p_1$ is another prime chosen so that $\ord_q((p_1-1)/12)>0$. But in the case $q >3$ it is even better to use a theorem of Ribet \cite[Theorem 2.3(2)]{Y} telling us that there exists a newform for $\Gamma_0(p_0p_2)$ satisfying the congruence, given that $q\mid(p_0+1)$. Here $p_2$ is {\em any} prime different from $p_0$ and $q$, and by choosing it so that $q\nmid(p_2^2-1)$, we can still consider the $q$ dividing $(p_0^2-1)$ to be the origin of the congruence.

\section{Example: $G=\GSp_2$, Klingen parabolic.}
Let
$$G=\GSp_2=\{g\in M_4 : g^tJg=\mu J,\mu\in \GL_1\}, \text{ where } J=\begin{pmatrix} 0_2 & -I_2\\I_2 & 0_2\end{pmatrix}.$$ It has a maximal torus $T=\{\diag(t_1,t_2,\mu t_1^{-1},\mu t_2^{-1}): t_1,t_2,\mu\in \GL_1\}$, with $X^*(T)$ spanned by $e_1,e_2$ and $e_0$, sending $\diag(t_1,t_2,\mu t_1^{-1},\mu t_2^{-1})$ to $t_1, t_2$ and $\mu$, respectively. The Weyl group $W_G$ is generated by permutations of the $e_i$ for $1\leq i\leq 2$ (with $e_0$ fixed), and inversions $e_i\mapsto e_0-e_i$, again for $1\leq i\leq 2$, with all other $e_j$ fixed. For $W_G$-invariant inner product on $X^*(T/S)\otimes\RR$ (those elements of $X^*(T)\otimes\RR$ such that the coefficient of $e_0$ is $-(1/2)$ times the sum of the other coefficients) we take the restriction of the bilinear form on $X^*(T)\otimes\RR$ such that $e_0$ is orthogonal to everything and $\langle e_i,e_j\rangle=\delta_{ij}$ for $1\leq i,j\leq n$. With a standard ordering, the positive roots are $\Phi^+=\{e_1-e_2,2e_1-e_0,e_1+e_2-e_0,2e_2-e_0\}$, with simple positive roots $\Delta_G=\{e_1-e_2,2e_2-e_0\}$, and $\rho_G=2e_1+e_2-(3/2)e_0$. In this section we choose $\alpha=e_1-e_2$, so $\Delta_M=\{2e_2-e_0\}$, $\Phi_N=\{e_1-e_2,e_1+e_2-e_0,2e_1-e_0\}$, $\rho_P=2e_1-e_0$, $\langle\rho_P,\check{\alpha}\rangle=2$ and $\tilde{\alpha}=e_1-\frac{1}{2}e_0$. Then we also find that $\Phi_N=\Phi_N^1$, i.e. $m=1$ and $r=r_1$. We have a Levi subgroup $M\simeq\GL_2\times\GL_1$, with
$$\left(A=\begin{pmatrix} a & b\\c & d\end{pmatrix},t\right)\mapsto\begin{pmatrix} t & 0 & 0 & 0\\ 0 & a & 0 & b\\0 & 0 & (\det A)t^{-1} & 0\\0 & c & 0 & d\end{pmatrix},$$ and $P=MN$ the Klingen parabolic, with unipotent radical
$$N=\left\{\begin{pmatrix} 1 & \gamma & \delta & \epsilon\\0 & 1 & \epsilon & 0\\0 & 0 & 1 & 0\\0 & 0 & -\gamma & 1\end{pmatrix}\right\},$$ and $(\gamma,\epsilon,\delta)\cdot(\gamma',\delta',\epsilon')=(\gamma+\gamma',\epsilon+\epsilon',\delta+\delta'+\gamma\epsilon'-\gamma'\epsilon).$

Let $f$ be a newform of weight $k'\geq 2$, trivial character, and $\Pi'$ the associated unitary, cuspidal, automorphic representation of $\GL_2(\A)$. Let $\Pi=\Pi'\times 1$, a unitary, cuspidal, automorphic representation of $M(\A)$. Since $\diag(T_1,T_2)\in\GL_2$ ends up as $\diag(1,T_1,T_1T_2,(T_1T_2)T_1^{-1})$ in $G$, the character of $\GL_2$ called `$e_1-e_2$' in the previous section becomes $2e_2-e_0\in X^*(T)$, so $\lambda=\left(\frac{k'-1}{2}\right)(2e_2-e_0).$ Similarly, at an unramified $p$ with $a_p(f)=p^{(k'-1)/2}(\alpha_p+\alpha_p^{-1})$ (recall that we have assumed the character to be trivial, for simplicity), we have $\chi_p=-\log_p(\alpha_p)(2e_2-e_0)$. (This is a logarithm to base $p$, not a $p$-adic logarithm.)
\vskip10pt
\begin{tabular}{|c|c|c|c|}\hline $\gamma\in\Phi_N$ & $\langle\lambda+s\tilde{\alpha},\check{\gamma}\rangle$ & $\langle 2e_2-e_0,\check{\gamma}\rangle$ & $|\chi_p(\check{\gamma}(p))|_p$\\\hline  $e_1-e_2$ & $s-(k'-1)$ & $-2$ & $\alpha_p^{-2}$\\ $e_1+e_2-e_0$ & $s+(k'-1)$ & $2$ & $\alpha_p^2$\\ $2e_1-e_0$ & $s$ & $0$ & $1$\\\hline \end{tabular}
\vskip10pt
Using the table, $L_p(s,\Pi_p,r)=(1-\alpha_p^2p^{-s})(1-\alpha_p^{-2}p^{-s})(1-p^{-s})$, and $L_{\Sigma}(s,\Pi,r)=L_{\Sigma}(\Sym^2 f,s+(k'-1))$.
We need $s\in\ZZ$ for $\lambda+s\tilde{\alpha}$ to be algebraically integral. (Look at the second column of the table.)
If $\MMM(f)$ is the motive of weight $k'-1$ attached to $f$, with Hodge type $\{(0,k'-1),(k'-1,0)\}$, then $F_{\infty}$ swaps $H^{(0,k'-1)}$ and $H^{(k'-1,0)}$, so acts as $+1$ on $H^{(k'-1,k'-1)}$ in the motive $\MMM(\Sym^2f)$ of weight $2k'-2$. It follows that if we want $t=0$ for $\MMM(r,\Pi\otimes |s\tilde{\alpha}|)(1)$ (c.f. Proposition \ref{critical}(2)) then we need $1+s+(k'-1)$ to be even, so $s$ is even. The minimum non-zero value of $\langle\lambda,\check{\gamma}\rangle$ is $b=k'-1$, so we are looking at even $s$ with $0<s\leq k'-2$.

We have that $\lambda+s\tilde{\alpha}=se_1+(k'-1)e_2-\frac{1}{2}(k'-1+s)e_0$. Let $w\in W_G$ switch $e_1$ and $e_2$, while leaving $e_0$ fixed, then $w(\lambda+s\tilde{\alpha})=(k'-1)e_1+se_2-\frac{1}{2}(k'-1+s)e_0$, which is (strictly) dominant, since $k'-1>s>0$. We recognise this as the infinitesimal character of $\Pi_{F,\infty}$, where $\Pi_F$ is a cuspidal automorphic representation of $G(\A)$ attached to a Siegel cusp form $F$ of genus $2$ and weight $\Sym^j\otimes\det^k$, if $k'-1=j+k-1$ and $s=k-2$ \cite[Theorem 3.1]{Mo}. Put another way, $j+k=k'$ and $1+s+(k'-1)=2k'-2-j$. Notice that $j=k'-s-2$ is even, and $k\geq 4$.

Dual to the basis $\{e_1,e_2,e_0\}$ of $X^*(T)$ is a basis $\{f_1,f_2,f_0\}$ of $X_*(T)$, with $f_1: t\mapsto\diag(t,1,t^{-1},1)$, $f_2:t\mapsto\diag(1,t,1,t^{-1})$ and $f_0: t\mapsto\diag(1,1,t,t)$. If we view $f_1+f_2+f_0$ as a character of $\hat{T}$, it is the highest weight of the $4$-dimensional spinor representation of $\hat{G}\simeq \GSp_2$. The complete set of weights is $\{f_1+f_2+f_0,f_1+f_0,f_2+f_0,f_0\}$. The element of $W_G$ switching $e_1$ and $e_2$ also switches $f_1$ and $f_2$, while fixing $f_0$. The element exchanging $e_1$ with $e_0-e_1$, while leaving $e_0$ and $e_2$ fixed, operates by switching the first and third elements of $\diag(t_1,t_2,\mu t_1^{-1},\mu t_2^{-1})$. So, while leaving $f_2$ fixed, it actually exchanges $f_0$ with $f_1+f_0$. Similarly we see that the element exchanging $e_2$ with $e_0-e_2$ fixes $f_1$ while exchanging $f_0$ with $f_2+f_0$. So $\{f_1+f_2+f_0,f_1+f_0,f_2+f_0,f_0\}$ is a single $W_G$-orbit, and $f_1+f_2+f_0$ is minuscule. Note that $(f_1+f_2+f_0)(p)=\diag(p,p,1,1)$, so $T_{f_1+f_2+f_0}$ is the usual genus-$2$ Hecke operator sometimes called ``$T(p)$''. Using $\chi_p+s\tilde{\alpha}=-\log_p(\alpha_p)(2e_2-e_0)+(k-2)(e_1-\frac{1}{2}e_0)$, we find
\vskip10pt
\begin{tabular}{|c|c|}\hline $\mu$ & $|(\chi_p+s\tilde{\alpha})(\mu(p))|_p$ \\ \hline $f_1+f_2+f_0$ & $\alpha_pp^{-(k-2)/2}$\\ $f_1+f_0$ & $\alpha_p^{-1}p^{-(k-2)/2}$\\ $f_2+f_0$ & $\alpha_pp^{(k-2)/2}$\\ $f_0$ & $\alpha_p^{-1}p^{(k-2)/2}$\\ \hline \end{tabular}
\vskip10pt
The trace is $p^{-(k-2)/2}(\alpha_p+\alpha_p^{-1})(1+p^{k-2})$. Multiplying by
$$p^{\langle\frac{k'-1}{2}(2e_2-e_0)+(k-2)(e_1-\frac{1}{2}e_0),f_1+f_2+f_0\rangle}=p^{(k'-1)/2}p^{(k-2)/2},$$
we find that
$$T_{f_1+f_2+f_0}(\Ind_P^G(\Pi_p\otimes |s\tilde{\alpha}|_p))=a_p(f)(1+p^{k-2}).$$ We recognise this as the eigenvalue of $T(p)$ on the vector-valued holomorphic Klingen-Eisenstein series $[f]$ of weight $\Sym^j\otimes\det^k$, though strictly speaking we would need $k>4$ to guarantee convergence, by \cite[Satz 1]{Kl}, \cite[Proposition 1.2]{A}. We are supposing, for simplicity, that $f$ has level $1$.

Suppose that $q>2\max\langle\lambda,\check{\gamma}\rangle +1=2(k'-1)+1=2k'-1$, and that
$$\ord_{\q}\left(\frac{L(\Sym^2 f,2k'-2-j)}{\Omega}\right)>0,$$
where $\q$ is a divisor of $q$ in a sufficiently large number field. (Here, $\Omega$ is a carefully scaled Deligne period, differing from the Petersson norm of $f$ by a congruence ideal and a power of $2\pi i$, see \cite[Lemma 5.1,(4)]{Du}, where it is called $(2\pi i)^{2(2k'-2-j)}\Omega$. It satisfies the requirement of \S 4, locally at $\q$.) Then Conjecture \ref{main} suggests a mod $\q$ congruence of Hecke eigenvalues between $[f]$ and a cuspidal eigenform of the same weight, and of level $1$ when $f$ is of level 1. Instances of such congruences were proved by Kurokawa and Mizumoto in the case $j=0$ (scalar weight) \cite{Ku,Mi}, by Satoh when $j=2$ \cite{Sa}, and in \cite[Proposition 4.4]{Du} for other $j$ as well. Using instead the pullback formula in \cite[\S 9]{Du} (from \cite[Proposition 4.4]{BSY}), a more general result could be proved, along the lines of what Katsurada and Mizumoto did for scalar weight \cite{KM}, as long as $k>5$. Possibly one can extend to $k=4$ using Hecke summation and analytic continuation.

We should also expect congruences of local origin, when we enlarge $\Sigma$ beyond the set of ramified primes for $\Pi$. For example, when $f$ has level 1 but we make $\Sigma=\{p_0\}$, then $(L_{p_0}(\Sym^2 f,2k'-2-j))^{-1}=p_0^{-(6k'-6-3j)}(p_0^{2k'-2-j}+2p_0^{k'-1}+p_0^j-a_{p_0}(f)^2)(p_0^{k-1}-1)$. Using data calculated as in \cite{BFvdG2}, we found experimental evidence for congruences between Hecke eigenvalues of $[f]$ (with $f$ of level 1) and cuspidal Hecke eigenforms for the principal genus-$2$ congruence subgroup of level~$2$. (In each case the congruence was checked for odd $p\leq 37$.) In all those cases for which the coefficient field is $\QQ$, and in which the modulus $q$ of the apparent congruence satisfies $q>2k'-1$, we checked that $q$ is indeed a divisor of $(p_0^{2k'-2-j}+2p_0^{k'-1}+p_0^j-a_{p_0}(f)^2)(p_0^{k-1}-1)$, with $p_0=2$. These cases are as follows.
\vskip10pt
\begin{tabular}{|c|c|c|}\hline $j$ & $k$ & $q$\\\hline $6$ & $6$ & $31$\\$4$ & $8$ & $41$\\$2$ & $10$ & $167$\\$0$ & $12$ & $89$\\ $10$ & $6$ & $53$\\ $0$ & $16$ & $733$\\ $0$ & $18$ & $967$\\ $4$ & $16$ & $113$\\\hline\end{tabular}
\vskip10pt
We also found a congruence for $(j,k,q)=(16,4,23)$, and in this case $q$ is not a divisor of $(p_0^{2k'-2-j}+2p_0^{k'-1}+p_0^j-a_{p_0}(f)^2)(p_0^{k-1}-1)$ with $p_0=2$ (nor of $\frac{L(\Sym^2 f,2k'-2-j)}{\Omega}$), but neither does it satisfy the condition $q>2k'-1$.

\section{Example: $G=\GSp_2$, Siegel parabolic.}
This time we choose $\alpha=2e_2-e_0$, so $\Delta_M=\{e_1-e_2\}$, $\Phi_N=\{2e_2-e_0,e_1+e_2-e_0,2e_1-e_0\}$, $\rho_P=\frac{3}{2}(e_1+e_2-e_0)$, $\langle\rho_P,\check{\alpha}\rangle=3/2$ and $\tilde{\alpha}=e_1+e_2-e_0$. Then we find that $m=2$, with $\Phi_N^1=\{2e_1-e_0,2e_2-e_0\}$ and $\Phi_N^2=\{e_1+e_2-e_0\}$. We have a Levi subgroup $M\simeq \GL_2\times\GL_1$, with
$$(A,\mu)\mapsto \begin{pmatrix} A & 0_2\\ 0_2 & \mu(A^t)^{-1}\end{pmatrix},$$
and $P=MN$ the Siegel parabolic, with unipotent radical
$$N=\left\{\begin{pmatrix} I_2 & B\\ 0_2 & I_2\end{pmatrix} : B^t=B\right\}.$$

Let $f$ be a newform of weight $k'\geq 2$, trivial character, and $\Pi'$ the associated unitary, cuspidal, automorphic representation of $\GL_2(\A)$. Let $\Pi=\Pi'\times 1$, which is a unitary, cuspidal, automorphic representation of $M(\A)$. Then $\lambda=\frac{k'-1}{2}(e_1-e_2)$ and $\chi_p=-\log_p(\alpha_p)(e_1-e_2)$.
\vskip10pt
\begin{tabular}{|c|c|c|c|}\hline $\gamma\in\Phi_N$ & $\langle\lambda+s\tilde{\alpha},\check{\gamma}\rangle$ & $\langle e_1-e_2,\check{\gamma}\rangle$ & $|\chi_p(\check{\gamma}(p))|_p$\\\hline  $2e_1-e_0$ & $s+\frac{k'-1}{2}$ & $1$ & $\alpha_p$\\ $2e_2-e_0$ & $s-\frac{k'-1}{2}$ & $-1$ & $\alpha_p^{-1}$\\ $e_1+e_2-e_0$ & $2s$ & $0$ & $1$\\\hline \end{tabular}
\vskip10pt
Using the table, $L_p(s,\Pi_p,r_1)=(1-\alpha_pp^{-s})(1-\alpha_p^{-1}p^{-s})$, and $L_{\Sigma}(s,\Pi,r_1)=L_{\Sigma}(f,s+\frac{k'-1}{2})$, while $L_p(s,\Pi_p,r_2)=(1-p^{-s})$, and $L_{\Sigma}(s,\Pi,r_2)=\zeta_{\Sigma}(s)$. We need $s\in\frac{1}{2}+\ZZ$ for $\lambda+s\tilde{\alpha}$ to be algebraically integral. As long as $s>0$, this also ensures that $L_{\Sigma}(1+2s,\Pi,r_2)$ is critical. For $L_{\Sigma}(1+s,\Pi,r_1)$ to be critical, we also need $s<\frac{k'-1}{2}$, and we exclude $s=1/2$.

We then have $\lambda+s\tilde{\alpha}=(\frac{k'-1}{2}+s)e_1+(s-\frac{k'-1}{2})e_2-se_0$. Let $w\in W_G$ switch $e_2$ and $e_0-e_2$, while leaving $e_1$ and $e_0$  fixed, then
$w(\lambda+s\tilde{\alpha})=(\frac{k'-1}{2}+s)e_1+(\frac{k'-1}{2}-s)e_2-\frac{k'-1}{2}e_0$, which is (strictly) dominant, since $0<s<\frac{k'-1}{2}$.
We recognise this as the infinitesimal character of $\Pi_{F,\infty}$, where $\Pi_F$ is a cuspidal automorphic representation of $G(\A)$ attached to a Siegel cusp form $F$ of genus $2$ and weight $\Sym^j\otimes\det^k$, if $j+k-1=\frac{k'-1}{2}+s$ and $k-2=\frac{k'-1}{2}-s$. So $s=\frac{j+1}{2}$ with $j\geq 0$ even, excluding $j=0$, and $k'=j+2k-2$. We find then that $L_{\Sigma}(1+2s,\Pi,r_2)=\zeta_{\Sigma}(j+2)$ and $L_{\Sigma}(1+s,\Pi,r_1)=L_{\Sigma}(f,j+k)$.

Using $\chi_p+s\tilde{\alpha}=-\log_p(\alpha_p)(e_1-e_2)+\frac{j+1}{2}(e_1+e_2-e_0)$, we find
\vskip10pt
\begin{tabular}{|c|c|}\hline $\mu$ & $|(\chi_p+s\tilde{\alpha})(\mu(p))|_p$ \\ \hline $f_1+f_2+f_0$ & $p^{-(j+1)/2}$\\ $f_1+f_0$ & $\alpha_p$\\ $f_2+f_0$ & $\alpha_p^{-1}$\\ $f_0$ & $p^{(j+1)/2}$\\ \hline \end{tabular}
\vskip10pt
The trace is $(\alpha_p+\alpha_p^{-1})+p^{(j+1)/2}+p^{-(j+1)/2}$. Multiplying by
$$p^{\langle (j+k-1)e_1+(k-2)e_2-\frac{j+2k-3}{2}e_0,f_1+f_2+f_0\rangle}=p^{(j+2k-3)/2}=p^{(k'-1)/2},$$
we find that
$$T_{f_1+f_2+f_0}(\Ind_P^G(\Pi_p\otimes |s\tilde{\alpha}|_p))=a_p(f)+p^{k-2}+p^{j+k-1}.$$

We begin with the case $i=1$. Suppose that $q>2\max\langle\lambda,\check{\gamma}\rangle +1=2\left(\frac{k'-1}{2}\right)+1=k'$, and that
$$\ord_{\q}\left(\frac{L(f,j+k)}{\Omega}\right)>0,$$
where $\q$ is a divisor of $q$ in a sufficiently large number field, and $\Omega=(2\pi i)^{j+k}\Omega^{(-1)^{j+k}}$ a Deligne period as in \cite[\S 2]{DIK}. Looking at Conjecture \ref{main}, with $f$ of level $1$ and $\Sigma=\emptyset$, then if $\tilde{\Pi}$ is the automorphic representation attached to a cuspidal Hecke eigenform $F$ (necessarily of level $1$) of weight $\Sym^j\otimes\det^k$, this becomes a conjecture of Harder \cite{H1}. There is ample experimental evidence for this conjecture, due to Faber and van der Geer, using \cite{FvdG}, as described in \cite{vdG}. The first example, also relayed in \cite{H1}, is with $(k',j,k,q)=(22,4,10,41)$.

One might object that Conjecture \ref{main} does not say that $\tilde{\Pi}$ should be attached to a cuspidal Hecke eigenform, i.e. that $\tilde{\Pi}$ should be holomorphic discrete series at $\infty$. But if it is not then it can be replaced by another $\tilde{\Pi}'$ that is, and which is the same as $\tilde{\Pi}$ at all finite places, as long as $\tilde{\Pi}$ is not CAP (which follows from $j>0$) or weakly endoscopic (which would follow from $\ord_q(B_{j+2})=0$). This is a consequence of a combination of two theorems of Weissauer \cite[Theorem 1]{We2}, \cite[Proposition 1.5]{We3}.

If we keep $f$ of level $1$, but enlarge $\Sigma$ to $\{p_0\}$, the conjecture demands congruences ``of local origin'', also first predicted by Harder \cite{H2}. We found experimentally (using data calculated as in \cite{BFvdG2}) several examples of such congruences (checked for odd $p\leq 37$), for $p_0=2$, with $F$ a Hecke eigenform for the principal congruence subgroup of level $2$. They are as follows. Note that if the Hecke eigenvalues of our eigenform $f$ are not defined over $\QQ$ then the congruence is only checked using norms. Note also that the example for $k'=30$ seems to work, though $q\not > k'$.
\vskip10pt
\begin{tabular}{|c|c|c|c|}\hline $k'$ & $j$ & $k$ & $q$\\\hline $16$ & $8$ & $5$ & $19$\\$18$ & $10$ & $5$ & $37$\\$26$ & $2$ & $13$ & $47$\\$30$ & $2$ & $15$ & $23$\\ $20$ & $14$ & $4$ & $61$\\\hline\end{tabular}
\vskip10pt

If we now let $f$ be a newform for $\Gamma_0(p_0)$, and keep $\Sigma=\{p_0\}$ at its minimum, then five apparent congruences are listed in \cite[\S 10]{BFvdG2}, for $p_0=2$, and we have found others since then, including for $\Gamma_0(4)$. D. Fretwell, a student of the second-named author, has used instead the method of \cite{Du2}, applied to algebraic modular forms on a compact form of $\GSp_2$, to find evidence for congruences with $p_0=2,3,5,7$, with $F$ a Hecke eigenform for the paramodular group at $p_0$.

Let us consider the excluded case $s=1/2$, i.e. $j=0$ (so $k'=2k-2$, and $F$ would be scalar-valued), and assume $\Sigma=\emptyset$ for simplicity. When $k$ is even, $a_p(f)+p^{k-1}+p^{k-2}$ is actually equal to the Hecke eigenvalue for the Saito-Kurokawa lift $\mathrm{SK}(f)$, a genus $2$, cuspidal Hecke eigenform of level $1$, weight $k$. Under weak conditions, Katsurada and Brown have independently proved a congruence modulo $\q$ of Hecke eigenvalues, between $\mathrm{SK}(f)$ and a non-lift eigenform $F$ \cite{Br,Ka}. Though $\Pi_{\mathrm{SK}(f)}$ is non-tempered, $\Pi_F$ should be tempered, so $\tilde{\Pi}=\Pi_F$ should satisfy the conjecture as stated (but without the exclusion of $s=1/2$). However, this will not work when $k$ is odd, when there is no Saito-Kurokawa lift. For example, when $k'=48$ (so $k=25$),
$$\ord_{\q}\left(\frac{L(f,k)}{\Omega}\right)>0,$$
with $q=7025111$, yet $S_k(\Sp_2(\ZZ))=\{0\}$ for odd $k<35$.

It seems likely that we can include $s=1/2$ within the scope of Conjecture \ref{main} if, as in \S 5, we drop our insistence that $\tilde{\Pi}_p$ should be unramified for $p\notin\Sigma$. Suppose that there exists a prime $p_0$ such that there exists a newform $g\in S_{k'}(\Gamma_0(p_0))$ with \begin{enumerate} \item $a_p(g)\equiv a_p(f)\pmod{\q}$ for all primes $p\nmid p_0q$;
\item $w_{p_0}(g)=-1$ (when $k$ is odd, $w_{p_0}$ being the eigenvalue of an Atkin-Lehner involution).
\end{enumerate}
 Fixing such a $p_0$ and such a $g$, in place of $\mathrm{SK}(f)$ we can put $\mathrm{Gk}(g)$, the Gritsenko lift \cite{Gri} of a Jacobi form corresponding to $g$ \cite{SZ}. This is a Hecke eigenform for the paramodular group at $p_0$, with Hecke eigenvalues $a_p(g)+p^{k-1}+p^{k-2}$ for $p\neq p_0$, and we would hope that there is a congruence modulo $\q$ of Hecke eigenvalues, between $\mathrm{Gk}(g)$ and a non-lift eigenform $F$. By \cite[Theorem A]{DT} there is a $g$ satisfying at least condition (1) if $a_{p_0}^2\equiv p_0^{k'-2}(1+p_0)^2\pmod{q}$. As in \cite[Lemma 7.1]{Ri2}, the infinitely many primes such that $a_{p_0}\equiv p_0+1\equiv 0\pmod{q}$ satisfy this condition. However, these are not really desirable for our purposes, since one easily checks that when $p_0^2\equiv 1\pmod{q}$ the congruence could be viewed as having local origin at $p_0$, i.e. $\ord_{q}(1-a_{p_0}p_0^{-k}+p_0^{k'-1-2k})>0$, whereas we would like to view it as originating from the $q$ in the complete $L$-value, with $p_0$ merely an auxiliary prime.

We must not forget the case $i=2$. Let's say $\Sigma=\emptyset$, so $f$ is of level $1$. Suppose
$$\ord_{\q}\left(\frac{\zeta(j+2)}{\pi^{j+2}}\right)>0,$$
with $q>2+2s=j+3$. Then, as noted in \S 5, there is a level $1$ cuspidal Hecke eigenform $g$ of weight $j+2$, such that $a_p(g)\equiv 1 +p^{j+1}\pmod{\q}$, for all primes $p$. If there were a genus $2$ Yoshida lift $Y(f,g)$, it would have the right weight $\Sym^j\otimes\det^k$ (in general $j=\wt(g)-2, k=\frac{\wt(f)-\wt(g)}{2}+2$), and the eigenvalue of $T_{f_1+f_2+f_0}$ would be $a_p(f)+p^{k-2}a_p(g)$, which is congruent modulo $\q$ to the right thing to satisfy the conjecture, namely $a_p(f)+p^{k-2}+p^{j+k-1}$. There may seem to be a problem with this, since the Yoshida lift does not exist at level $1$. But the endoscopic lift of $\Pi_f$ and $\Pi_g$ still exists as an automorphic representation, with the same Hecke eigenvalues, and this is all we need to satisfy the conjecture. Indeed, the conjecture does not state that $\tilde{\Pi}$ should have holomorphic vectors. (Here $\tilde{\Pi}_{\infty}$ is non-holomorphic discrete series, with a Whittaker model, and Harish-Chandra parameter $(j+k-1)e_1-(k-2)e_2-\frac{j+1}{2}e_0$, see \cite[top of p.8]{Mo}.)

\section{Example: $G=\GSp_3$, $M\simeq \GL_2\times\GL_2$.}
Let
$$G=\GSp_3=\{g\in M_6 : g^tJg=\mu J\}, \text{ where } J=\begin{pmatrix} 0_3 & -I_3\\I_3 & 0_3\end{pmatrix}.$$
It has maximal torus $T=\{\diag(t_1,t_2,t_3,\mu t_1^{-1},\mu t_2^{-1},\mu t_3^{-1}): t_1,t_2,t_3,\mu\in \GL_1\}$, with $X^*(T)$ spanned by $e_1,e_2,e_3$ and $e_0$, sending $\diag(t_1,t_2,t_3,\mu t_1^{-1},\mu t_2^{-1},\mu t_3^{-1})$ to $t_1, t_2, t_3$ and $\mu$, respectively. The Weyl group $W_G$ is generated by permutations of the $e_i$ for $1\leq i\leq 3$ (with $e_0$ fixed), and inversions $e_i\mapsto e_0-e_i$, again for $1\leq i\leq 3$, with all other $e_j$ fixed. For $W_G$-invariant inner product on $X^*(T/S)\otimes\RR$ (those elements of $X^*(T)\otimes\RR$ such that the coefficient of $e_0$ is $-(1/2)$ times the sum of the other coefficients) we take the restriction of the bilinear form on $X^*(T)\otimes\RR$ such that $e_0$ is orthogonal to everything and $\langle e_i,e_j\rangle=\delta_{ij}$ for $1\leq i,j\leq 3$. With a standard ordering, the positive roots are $\Phi^+=\{e_1-e_2, e_1-e_3, e_2-e_3, 2e_1-e_0, 2e_2-e_0, 2e_3-e_0,e_1+e_2-e_0,e_1+e_3-e_0,e_2+e_3-e_0\}$, with simple positive roots $\Delta_G=\{e_1-e_2, e_2-e_3, 2e_3-e_0\}$, and $\rho_G=3e_1+2e_2+e_3-3e_0$. In this section we choose $\alpha=e_2-e_3$, so $\Delta_M=\{e_1-e_2, 2e_3-e_0\}$, $\Phi_N=\{e_1-e_3, e_2-e_3, 2e_1-e_0, 2e_2-e_0, e_2+e_3-e_0, e_3+e_1-e_0, e_1+e_2-e_0\}$, $\rho_P=\frac{5}{2}(e_1+e_2-e_0)$, $\langle\rho_P,\check{\alpha}\rangle=5/2$ and $\tilde{\alpha}=e_1+e_2-e_0$, then we find also $m=2$, with $\Phi_N^2=\{e_1+e_2-e_0\}$.

We have a Levi subgroup $M\simeq\GL_2\times\GSp_1\simeq \GL_2\times\GL_2$, with
$$\left(A,\begin{pmatrix} a & b\\c & d\end{pmatrix}\right)\mapsto\begin{pmatrix} A & & & \\ & a & & b\\ & & (A^t)^{-1}\mu\begin{pmatrix} a & b\\c & d\end{pmatrix}& \\ & c & & d\end{pmatrix}.$$

Let $f$ and $g$ be newforms of weights $k$ and $\ell$, respectively. Let $\Pi_f$ and $\Pi_g$ be the associated unitary, cuspidal, automorphic representations of $\GL_2(\A)$. Say, for unramified $p$, that $a_p(f)=p^{(k-1)/2}(\alpha_p+\alpha_p^{-1})$ and $a_p(g)=p^{(\ell-1)/2}(\beta_p+\beta_p^{-1})$ (so we are assuming trivial character, for simplicity). For $M\simeq\GL_2\times\GL_2$, let $\Pi=\Pi_f\times\Pi_g$, then $\lambda=\frac{k-1}{2}(e_1-e_2)+\frac{\ell-1}{2}(2e_3-e_0)$ and $\chi_p=-\log_p(\alpha_p)(e_1-e_2)-\log_p(\beta_p)(2e_3-e_0)$.
\vskip10pt
\begin{tabular}{|c|c|c|c|c|}\hline $\gamma\in\Phi_N$ & $\langle e_1-e_2,\check{\gamma}\rangle$ & $\langle 2e_3-e_0,\check{\gamma}\rangle$ & $\langle\lambda+s\tilde{\alpha},\check{\gamma}\rangle$ & $|\chi_p(\check{\gamma}(p))|_p$\\\hline $e_1-e_3$ & $1$ & $-2$ & $\frac{k-1}{2}-(\ell-1)+s$ & $\alpha_p\beta_p^{-2}$\\ $e_2-e_3$ & $-1$ & $-2$ & $-\frac{k-1}{2}-(\ell-1)+s$ & $\alpha_p^{-1}\beta_p^{-2}$\\ $2e_1-e_0$ & $1$ & $0$ & $\frac{k-1}{2}+s$ & $\alpha_p$\\ $2e_2-e_0$ & $-1$ & $0$ & $-\frac{k-1}{2}+s$ & $\alpha_p^{-1}$\\ $e_1+e_3-e_0$ & $1$ & $2$ & $\frac{k-1}{2}+(\ell-1)+s$ & $\alpha_p\beta_p^2$\\ $e_2+e_3-e_0$ & $-1$ & $2$ & $-\frac{k-1}{2}+(\ell-1)+s$ & $\alpha_p^{-1}\beta_p^2$\\$e_1+e_2-e_0$ & $0$ & $0$ & $2s$ & $1$\\\hline \end{tabular}
\vskip10pt
Using the table, $L_{\Sigma}(s,\Pi,r_1)=L_{\Sigma}(s+(\ell-1)+\frac{k-1}{2},\Sym^2(g)\otimes f)$ and $L_{\Sigma}(s,\Pi,r_2)=\zeta_{\Sigma}(s)$. For $\lambda+s\tilde{\alpha}$ to be algebraically integral, we need $s\in \frac{1}{2}+\ZZ$. We have
$$\lambda+s\tilde{\alpha}=\frac{k-1}{2}(e_1-e_2)+\frac{\ell-1}{2}(2e_3-e_0)+s(e_1+e_2-e_0)$$
$$=\left(\frac{k-1}{2}+s\right)e_1+\left(s-\frac{k-1}{2}\right)e_2+(\ell-1)e_3-\left(s+\frac{\ell-1}{2}\right)e_0.$$
Note that $\nu_1e_1+\nu_2e_2+\nu_3e_3-\left(\frac{\nu_1+\nu_2+\nu_3}{2}\right)e_0$ is strictly dominant (with respect to $\Delta_G$) if and only if $\nu_1>\nu_2>\nu_3>0$. We must apply elements of $W_G$ to make this happen.

{\bf Case 1: $\frac{k-1}{2}>\ell-1$.}
It is easy to find the right $w$, exchanging $e_2$ and $e_0-e_2$ while fixing $e_0,e_1$ and $e_3$. Then
$$w(\lambda+s\tilde{\alpha})=\left(\frac{k-1}{2}+s\right)e_1+\left(\frac{k-1}{2}-s\right)e_2+(\ell-1)e_3-\left(\frac{k-1}{2}+\frac{\ell-1}{2}\right)e_0.$$
As expected, this is strictly dominant for $0<s<b_1=\frac{k-1}{2}-(\ell-1)$. A genus $3$ Siegel cuspidal Hecke eigenform gives a unitary, cuspidal, automorphic representation of $\GSp_3(\A)$ whose component at $\infty$ has infinitesimal character $(a+3)e_1+(b+2)e_2+(c+1)e_3-\left(\frac{a+b+c+6}{2}\right)e_0$, where $a\geq b\geq c$, and $(a-b,b-c,c+4)$ is the ``weight'' of the form \cite[\S\S3,4]{BFvdG1}. So we put $a+3=\frac{k-1}{2}+s$, $b+2=\frac{k-1}{2}-s$ and $c+1=\ell-1$, i.e. $k=a+b+6, \ell=c+2$, $s=\frac{a-b+1}{2}$ (necessitating $a\equiv b\pmod{2})$. Excluding $s=\frac{1}{2}$ excludes $a=b$. We find then that $L_{\Sigma}(1+2s,\Pi,r_2)=\zeta_{\Sigma}(a-b+2)$ and $L_{\Sigma}(1+s,\Pi,r_1)=L_{\Sigma}(\Sym^2(g)\otimes f,a+c+5)$.

Dual to the basis $\{e_1,e_2,e_3,e_0\}$ of $X^*(T)$ is a basis $\{f_1,f_2,f_3,f_0\}$ of $X_*(T)$, with $f_1: t\mapsto\diag(t,1,1,t^{-1},1,1)$, $f_2:t\mapsto\diag(1,t,1,1,t^{-1},1)$, $f_3:t\mapsto\diag(1,1,t,1,1,t^{-1})$ and $f_0: t\mapsto\diag(1,1,1,t,t,t)$. If we view $f_1+f_2+f_3+f_0$ as a character of $\hat{T}$, one easily checks, as in \S 6, that it is minuscule, with $W_G$-orbit as in the table below. It is the highest weight of the $8$-dimensional spinor representation of $\hat{G}$.
Using $\chi_p+s\tilde{\alpha}=-\log_p(\alpha_p)(e_1-e_2)-\log_p(\beta_p)(2e_3-e_0)+\frac{a-b+1}{2}(e_1+e_2-e_0)$, we find
\vskip10pt
\begin{tabular}{|c|c|}\hline $\mu$ & $|(\chi_p+s\tilde{\alpha})(\mu(p))|_p$ \\ \hline $f_1+f_2+f_3+f_0$ & $\beta_pp^{-\frac{a-b+1}{2}}$\\ $f_1+f_2+f_0$ & $\beta_p^{-1}p^{-\frac{a-b+1}{2}}$\\$f_1+f_3+f_0$ & $\alpha_p\beta_p$\\$f_1+f_0$ & $\alpha_p\beta_p^{-1}$\\$f_2+f_3+f_0$ & $\alpha_p^{-1}\beta_p$\\$f_2+f_0$ & $\alpha_p^{-1}\beta_p^{-1}$\\$f_3+f_0$ & $\beta_pp^{\frac{a-b+1}{2}}$\\$f_0$ & $\beta_p^{-1}p^{\frac{a-b+1}{2}}$\\ \hline \end{tabular}
\vskip10pt
The trace is $(\beta_p+\beta_p^{-1})(\alpha_p+\alpha_p^{-1})+p^{(a-b+1)/2}+p^{-(a-b+1)/2}$. Multiplying by $p^{\langle (a+3)e_1+(b+2)e_2+(c+1)e_3-\frac{a+b+c+6}{2}e_0,f_1+f_2+f_3+f_0\rangle}=p^{(a+b+c+6)/2}=p^{(k-1)/2}p^{(\ell-1)/2}$, we find that
$$T_{f_1+f_2+f_3+f_0}(\Ind_P^G(\Pi_p\otimes |s\tilde{\alpha}|_p))=a_p(g)(a_p(f)+p^{a+3}+p^{b+2}).$$ Note that $(f_1+f_2+f_3+f_0)(p)=\diag(p,p,p,1,1,1)$.

Suppose that $q>2\max\langle\lambda,\check{\gamma}\rangle +1=2\left(\frac{k-1}{2}+(\ell-1)\right)+1=k+2\ell-2=a+b+2c+8$, and that
$$\ord_{\q}\left(\frac{L(\Sym^2(g)\otimes f,a+c+5)}{\Omega}\right)>0,$$
where $\q$ is a divisor of $q$ in a sufficiently large number field. Looking at Conjecture~\ref{main}, in the case $i=1$, if $f$ and $g$ are of level $1$ and $\Sigma=\emptyset$, if $\tilde{\Pi}$ is the automorphic representation attached to a cuspidal Hecke eigenform $F$ (necessarily of level $1$) of weight $(a-b,b-c,c+4)$, this becomes Conjecture 10.10 in \cite{BFvdG1}, which they worked out in collaboration with Harder. See \cite{H5} for a route to the conjecture (and that in Case 2) somewhat different from the above. They showed that a congruence of the right shape holds for $p\leq 17$, with $(a,b,c)=(13,11,10)$ and $q=199$. The norm of the ratio of the $L$-value to another (making the problematic periods cancel) was approximated numerically by A. Mellit, who found $199$ in the numerator.

{\bf Case 2: $\ell-1>\frac{k-1}{2}$.}
Recall that $$\lambda+s\tilde{\alpha}=\left(\frac{k-1}{2}+s\right)e_1+\left(s-\frac{k-1}{2}\right)e_2+(\ell-1)e_3-\left(s+\frac{\ell-1}{2}\right)e_0.$$
This time we choose $w:e_1\mapsto e_2, e_2\mapsto e_0-e_3, e_3\mapsto e_1$ and $e_0\mapsto e_0$, to get
$$w(\lambda+s\tilde{\alpha})=(\ell-1)e_1+\left(\frac{k-1}{2}+s\right)e_2+\left(\frac{k-1}{2}-s\right)e_3-\left(\frac{\ell-1}{2}+\frac{k-1}{2}\right)e_0.$$
Now $a+3=\ell-1$, $b+2=\frac{k-1}{2}+s$, $c+1=\frac{k-1}{2}-s$, i.e. $\ell=a+4$, $k=b+c+4$, $s=\frac{b-c+1}{2}$. We need $0<s<\ell-1-\frac{k-1}{2}$, with $s\in \frac{1}{2}+\ZZ$ (so $b\equiv c\pmod{2}$). We find then that $L_{\Sigma}(1+2s,\Pi,r_2)=\zeta_{\Sigma}(b-c+2)$ and $L_{\Sigma}(1+s,\Pi,r_1)=L_{\Sigma}(\Sym^2(g)\otimes f,a+b+6)$. By the same method as in Case 1, one finds that $$T_{f_1+f_2+f_3+f_0}(\Ind_P^G(\Pi_p\otimes |s\tilde{\alpha}|_p))=a_p(g)(a_p(f)+p^{b+2}+p^{c+1}).$$ This time Conjecture \ref{main} (for $i=1$, $f$ and $g$ of level $1$ and $\Sigma=\emptyset$) becomes Conjecture 10.8 in \cite{BFvdG1}. They confirmed (for $p\leq 17$) seventeen congruences of this shape, including for instance $(a,b,c)=(12,6,2)$ and $q=101$. Again, the right primes were found numerically in ratios of $L$-values by Mellit, and this was supported also by algebraic calculations of the $L$-values, using triple product $L$-functions, see \cite[Table 3]{IKPY}.

To exclude $s=\frac{1}{2}$ we need to exclude $b=c$. In the case $b=c$, according to \cite[Conjecture 7.7(ii)]{BFvdG1} there should exist a ``lift'', a genus $3$ Siegel cuspidal eigenform of weight $(a-b,0,b+4)$ (and level $1$), with eigenvalue of $T_{f_1+f_2+f_3+f_0}$ actually equal to $a_p(g)(a_p(f)+p^{b+2}+p^{c+1})$. Then $\q$ as above should be the modulus of a congruence of Hecke eigenvalues between this lift and some non-lift cuspidal eigenform in the same space. In the scalar-valued case $a=b=c$ (and all even), such a lift was conjectured by Miyawaki \cite{Miy}, and its existence proved by Ikeda \cite{Ik} (with $f$, $g$ of level $1$). Ibukiyama, Katsurada, Poor and Yuen {\em proved} two instances of such congruences in \cite[\S 5]{IKPY}, with $(\ell,k,q)=(16,28,107)$ and $(20,36,157)$, using pullback formulas. They also found $\q$ of norm $q$ in the ``algebraic part'' of $L_{\Sigma}(\Sym^2(g)\otimes f,2a+6)$, \cite[Table 2]{IKPY}. One can check that the period they divided by is, locally at $\q$, of the type required in Conjecture \ref{main}. Similarly in Case~1 with $a=b$, there should be a lift, according to \cite[Conjecture 7.7(iii)]{BFvdG1}, again generalising a conjecture of Miyawaki in the scalar valued case, and it would be natural to conjecture congruences between lifts and non-lifts.

When $i=2$, let's say with $\Sigma=\emptyset$, we are looking, in Case 1, at large $\q\mid\frac{\zeta(a-b+2)}{\pi^{a-b+2}}$, and we know there will be a level $1$ cuspidal Hecke eigenform $h$ of weight $a-b+2$ such that $a_p(h)\equiv 1+p^{a-b+1}\pmod{\q}$, for all primes $p$. Then $a_p(g)(a_p(f)+p^{a+3}+p^{b+2})\equiv a_p(g)(a_p(f)+p^{b+2}a_p(h))$, which would be the Hecke eigenvalue of a conjectural endoscopic lift with standard $L$-function $L(\Sym^2(g),s+c+1)L(f\otimes h, s+a+3)$. This endoscopic lift would supply the $\tilde{\Pi}$ required by Conjecture \ref{main}. It should exist as an automorphic representation, but without holomorphic vectors. Its contribution to cohomology appears in \cite[Conjecture 7.12]{BFvdG1}, as the second of the two terms.

In Case 2 we are looking at large $\q\mid\frac{\zeta(b-c+2)}{\pi^{b-c+2}}$, and we know there will be a level $1$ cuspidal Hecke eigenform $h$ of weight $b-c+2$ such that $a_p(h)\equiv 1+p^{b-c+1}\pmod{\q}$, for all primes $p$. Then $a_p(g)(a_p(f)+p^{b+2}+p^{c+1})\equiv a_p(g)(a_p(f)+p^{c+1}a_p(h))$, which should be the Hecke eigenvalue of an endoscopic lift corresponding to the other term in \cite[Conjecture 7.12]{BFvdG1}, this time with standard $L$-function $L(\Sym^2(g),s+a+3)L(f\otimes h, s+b+2)$ (and appearing also in \cite[Conjecture 10.12]{BFvdG1}).

\section{Example: $G=\GSp_3$, $M\simeq \GL_1\times\GSp_2$.}
Now choose $\alpha=e_1-e_2$, so $\Delta_M=\{e_2-e_3, 2e_3-e_0\}$, $\Phi_N=\Phi_N^1=\{e_1-e_2, e_1-e_3, e_1+e_2-e_0, e_1+e_3-e_0, 2e_1-e_0\}$, $\rho_P=3e_1-\frac{3}{2}e_0$, $\langle\rho_P,\check{\alpha}\rangle=3$ and $\tilde{\alpha}=e_1-\frac{1}{2}e_0$. We have a Levi subgroup $M\simeq \GL_1\times\GSp_2$, with
$$\left(a,\begin{pmatrix}A & B\\C & D\end{pmatrix}\right)\mapsto \begin{pmatrix} a & 0 & 0 & 0\\0 & A & 0 & B\\0 & 0 & a^{-1}\mu\left(\begin{pmatrix}A & B\\C & D\end{pmatrix}\right) & 0\\0 & C & 0 & D\end{pmatrix}.$$

Let $\Pi=1\times\Pi_F$, where $\Pi_F$ is a unitary, cuspidal, automorphic representation of $\GSp_2(\A)$ associated with a genus $2$ cuspidal Hecke eigenform $F$, of weight $\Sym^j\otimes\det^k$. At an unramified $p$, let $\alpha_0, \alpha_2,\alpha_3$ be the Satake parameters for $\Pi_F$, then $\lambda=(j+k-1)e_2+(k-2)e_3-\frac{j+2k-3}{2}e_0$ and $\chi_p=-[\log_p(\alpha_2)e_2+\log_p(\alpha_3)e_3+\log_p(\alpha_0)e_0]$. Note that $\alpha_0^2\alpha_2\alpha_3=1$.
\vskip10pt
\begin{tabular}{|c|c|c|c|c|}\hline $\gamma\in\Phi_N$ & $\langle\lambda+s\tilde{\alpha},\check{\gamma}\rangle$ & $|\chi_p(\check{\gamma}(p))|_p$\\\hline $e_1-e_2$ & $-(j+k-1)+s$ & $\alpha_2^{-1}$\\ $e_1-e_3$ & $-(k-2)+s$ & $\alpha_3^{-1}$\\$e_1+e_2-e_0$ & $(j+k-1)+s$ & $\alpha_2$\\$e_1+e_3-e_0$ & $(k-2)+s$ & $\alpha_3$\\ $2e_1-e_0$ & $s$ & $1$ \\\hline \end{tabular}
\vskip10pt
Using the table, $L_{\Sigma}(s,\Pi,r)=L_{\Sigma}(s, F, \mathrm{st})$, which is the standard $L$-function. For $\lambda+s\tilde{\alpha}$ to be algebraically integral, we need $s\in \ZZ$. Since there is a $\gamma$ such that $\langle\lambda,\check{\gamma}\rangle=0$, there must also be a parity condition. Consideration of the relation between the standard representation and the exterior square of the spinor representation of $\hat{G}$ shows that it should be $s\in 1+2\ZZ$.
$$\lambda+s\tilde{\alpha}=(j+k-1)e_2+(k-2)e_3-\frac{j+2k-3}{2}e_0+s(e_1-\frac{1}{2}e_0)$$
$$=se_1+(j+k-1)e_2+(k-2)e_3-\frac{j+2k+s-3}{2}e_0.$$
Choose $w:e_1\mapsto e_3\mapsto e_2\mapsto e_1$ and $e_0\mapsto e_0$, then
$$w(\lambda+s\tilde{\alpha})=(j+k-1)e_1+(k-2)e_2+se_3-\frac{j+2k+s-3}{2}e_0,$$
so $(a,b,c)=(j+k-4,k-4,s-1)$. We have $s\in 1+2\ZZ$ and $0<s<k-2$, excluding $s=1$, and $L_{\Sigma}(1+s,\Pi,r)=L_{\Sigma}(c+2,F,\st)$.

Using $\chi_p=-[\log_p(\alpha_2)e_2+\log_p(\alpha_3)e_3+\log_p(\alpha_0)e_0]$, we find
\vskip10pt
\begin{tabular}{|c|c|}\hline $\mu$ & $|(\chi_p+s\tilde{\alpha})(\mu(p))|_p$ \\ \hline $f_1+f_2+f_3+f_0$ & $\alpha_0\alpha_2\alpha_3p^{-s/2}$\\ $f_1+f_2+f_0$ & $\alpha_0\alpha_2p^{-s/2}$\\$f_1+f_3+f_0$ & $\alpha_0\alpha_3p^{-s/2}$\\$f_1+f_0$ & $\alpha_0p^{-s/2}$\\$f_2+f_3+f_0$ & $\alpha_0\alpha_2\alpha_3p^{s/2}$\\$f_2+f_0$ & $\alpha_0\alpha_2p^{s/2}$\\$f_3+f_0$ & $\alpha_0\alpha_3p^{s/2}$\\$f_0$ & $\alpha_0p^{s/2}$\\ \hline \end{tabular}
\vskip10pt
The trace is $(\alpha_0+\alpha_0\alpha_2+\alpha_0\alpha_3+\alpha_0\alpha_2\alpha_3)p^{-s/2}(1+p^s)$. Multiplying by
$$p^{(a+b+c+6)/2}=p^{\frac{j+2k-3+s}{2}},$$
we find that
$$T_{f_1+f_2+f_3+f_0}(\Ind_P^G(\Pi_p\otimes |s\tilde{\alpha}|_p))=T(p)(\Pi_F)(1+p^{c+1}),$$
which is the Hecke eigenvalue for a holomorphic genus $3$ Klingen-Eisenstein series attached to $F$, so as in \S 6, we are looking at congruences between Klingen-Eisenstein series and cusp forms.
Suppose that $q>2\max\langle\lambda,\check{\gamma}\rangle +1=2(j+k-1)+1$, equivalently $q>2(j+k)$, and that
$$\ord_{\q}\left(\frac{L(c+2,F,\st)}{\Omega}\right)>0,$$
where $\q$ is a divisor of $q$ in a sufficiently large number field. Looking at Conjecture \ref{main}, if $F$ is level $1$ and $\Sigma=\emptyset$, if $\tilde{\Pi}$ is the automorphic representation attached to a cuspidal Hecke eigenform $G$ (necessarily of level $1$) of weight $(a-b,b-c,c+4)$, \cite[Table 5]{BFvdG1} gives $12$ experimental congruences of this shape, for various $q$, but without the link to $L(c+2,F,\st)$ (or any conjecture about where the modulus comes from). We can make this link in three cases, namely $(a,b,c)=(15,5,4)$ with $q=29$, $(a,b,c)=(10,6,4)$ with $q=41$, and $(a,b,c)=(16,16,16)$ with $q=691$.

In the first two cases, $q$ features in Harder's conjecture (see \S 7), for $(j,k)=(a-b,b+4)$, and
$$\ord_{q}\left(\frac{L(f,j+k)}{\Omega'}\right)>0,$$
for $f$ of weight $k'=j+2k-2$. In \cite{DIK} it is explained, using the Bloch-Kato conjecture, how this should lead to divisibility by $\q$ of $\frac{L(j/2+1,F,\st)}{\Omega}$ (Conjecture 5.4, when $k'/2$ and $j/2$ are odd and $(j/2)+1\leq k-2$) and of $\frac{L(j+2,F,\st)}{\Omega}$ (Conjecture 5.3, when $j\leq k-4$). The numbers $(j/2)+1$ and $j+2$ coincide with $c+2$ in these first two examples, for $q=29$ and $q=41$ respectively. (Strictly speaking, though $29>k'$, $29\not >2(j+k)$.) For the $q=41$ case, the divisibility by $q$, of $\frac{\pi^6L(6,F,\st)}{L(8,F,\st)}$, is actually proved \cite[Corollary 7.12]{DIK}. In the case $(a,b,c)=(16,16,16)$, $q=691$, the divisibility is proved in Example 3 at the end of \cite{KM}.

A further example, which appears however in \cite[Table 4]{BFvdG1}, is $(a,b,c)=(15,5,4)$ with $q=691$. Here the congruence is of the form
$$T_{f_1+f_2+f_3+f_0}(\Pi_G)\equiv (a_p(f)+p^{b+2}a_p(g))(1+p^{c+1})\pmod{\q},$$ with $f$ and $g$ cuspidal Hecke eigenforms of level $1$ and with weights $a+b+6$ and $a-b+2$ respectively ($26$ and $12$ in this example, and the congruence is checked using norms). If we let $\Pi'$ be the endoscopic lift of $\Pi_f$ and $\Pi_g$ (which does not have a holomorphic vector $F$) then this congruence is of the same type as above, since $T(p)(\Pi')=a_p(f)+p^{b+2}a_p(g)$. Now $L(s,\Pi',\st)=\zeta(s)L(f\otimes g,s+a+3)$, so with $a=15$ and $s+1=c+2=6$, it suffices to show that $691$ divides a suitably normalised $L(f\otimes g,24)$, but this is a consequence of \cite[Theorem 14.2]{Du3}. This example is actually analogous to the first one above. How the Bloch-Kato conjecture leads one to expect the divisibility is explained in \cite[\S\S 8,11]{Du3}.

\section{Example: $G=\GSp_3$, $M\simeq\GL_1\times\GL_3$.}
Now choose $\alpha=2e_3-e_0$, so $\Delta_M=\{e_1-e_2, e_2-e_3\}$, $\Phi_N^1=\{2e_1-e_0,2e_2-e_0,2e_3-e_0\}$, $\Phi_N^2=\{e_1+e_2-e_0, e_1+e_3-e_0, e_2+e_3-e_0\}$, $\tilde{\alpha}=e_1+e_2+e_3-\frac{3}{2}e_0$. We have a Levi subgroup $M\simeq \GL_1\times\GL_3$, with
$$(a,A)\mapsto \begin{pmatrix} A & 0_3\\0_3 & a(A^t)^{-1}\end{pmatrix}.$$

Let $\Pi=1\times\Pi'$, where $\Pi'$ is a unitary, cuspidal, automorphic representation of $\GL_3(\A)$. From now on we look only at the case that $\Pi'$ is the symmetric square lifting of $\Pi_f$, where $f$ is a cuspidal Hecke eigenform of genus $1$ and weight $k$, trivial character, with $a_p(f)=p^{(k-1)/2}(\alpha_p+\alpha_p^{-1})$. (Actually, when the level is $1$, this is the only possibility, according to \cite[Conjecture 1.1]{AP}.) Then $\lambda=(k-1)(e_1-e_3)$ and $\chi_p=-\log_p(\alpha_p)(2e_1-2e_3)$.
\vskip10pt
\begin{tabular}{|c|c|c|}\hline $\gamma\in\Phi_N$ & $\langle\lambda+s\tilde{\alpha},\check{\gamma}\rangle$ & $|\chi_p(\check{\gamma}(p))|_p$\\\hline $2e_1-e_0$ & $s+(k-1)$ & $\alpha_p^2$\\ $2e_2-e_0$ & $s$ & $1$\\ $2e_3-e_0$ & $s-(k-1)$ & $\alpha_p^{-2}$\\ $e_1+e_2-e_0$ & $2s+(k-1)$ & $\alpha_p^2$\\ $e_1+e_3-e_0$ & $2s$ & $1$ \\ $e_2+e_3-e_0$ & $2s-(k-1)$ & $\alpha_p^{-2}$\\\hline \end{tabular}
\vskip10pt
Using the table, $L_{\Sigma}(s,\Pi,r_1)=L_{\Sigma}(s,\Pi,r_2)=L_{\Sigma}(\Sym^2 f,s+(k-1))$ For $\lambda+s\tilde{\alpha}$ to be algebraically integral, we need $s\in \ZZ$. We have $L(1+s,\Pi,r_1)$ critical for $0<s<k-1$ with $s\in 2\ZZ$, $L(1+2s,\Pi,r_2)$ critical for $0<s<\frac{k-1}{2}$ with $s\in \ZZ$, both critical for $0<s<\frac{k-1}{2}$ with $s\in 2\ZZ$
$$\lambda+s\tilde{\alpha}=(k-1+s)e_1+se_2+(s-(k-1))e_3-\frac{3}{2}se_0.$$

Choosing $w$ appropriately,
$$w(\lambda+s\tilde{\alpha})=(k-1+s)e_1+(k-1-s)e_2+se_3-(k-1+(s/2))e_0,$$
which is strictly dominant for $0<s<\frac{k-1}{2}$. We recognise it as the infinitesimal character of $\Pi_G$ with $(a,b,c)=(k-1+s-3, k-1-s-2,s-1)$.
The requirement that $s\in 2\ZZ$ has the desirable effect that $a+b+c$ is even, which is necessary to avoid the space of genus $3$ cuspforms being trivial \cite[Remark 4.2]{BFvdG1}.

Using $\chi_p=-\log_p(\alpha_p)(2e_1-2e_3)$, we find
\vskip10pt
\begin{tabular}{|c|c|}\hline $\mu$ & $|(\chi_p+s\tilde{\alpha})(\mu(p))|_p$ \\ \hline $f_1+f_2+f_3+f_0$ & $p^{-3s/2}$\\ $f_1+f_2+f_0$ & $\alpha_p^2p^{-s/2}$\\$f_1+f_3+f_0$ & $p^{-s/2}$\\$f_1+f_0$ & $\alpha_p^2p^{s/2}$\\$f_2+f_3+f_0$ & $\alpha_p^{-2}p^{-s/2}$\\$f_2+f_0$ & $p^{s/2}$\\$f_3+f_0$ & $\alpha_p^{-2}p^{s/2}$\\$f_0$ & $p^{3s/2}$\\ \hline \end{tabular}
\vskip10pt
The trace is $p^{-3s/2}+p^{3s/2}+(\alpha_p^2+1+\alpha_p^{-2})(p^{-s/2}+p^{s/2})$. Multiplying by
$$p^{(a+b+c+6)/2}=p^{k-1+(s/2)},$$
we find that
$$T_{f_1+f_2+f_3+f_0}(\Ind_P^G(\Pi_p\otimes |s\tilde{\alpha}|_p))=p^{k-1-s}+p^{k-1+2s}+(a_p(f)^2-p^{k-1})(1+p^s)$$
$$=p^{b+2}+p^{a+c+4}+(a_p(f)^2-p^{(a+b+5)/2})(1+p^{c+1}).$$

The smallest example (with level $1$ and $\Sigma=\emptyset$) where we might hope to test the congruence is $k=16$, $s=4$, $q=2243$, so $(a,b,c)=(16,9,3)$. Unfortunately, the dimension of the space of genus $3$ cuspforms of this type has dimension $4$, which is prohibitively large.

\section{Example: $G=G_2$, omitting the short root.}
Let $G$ be the Chevalley group of type $G_2$. Then $\Delta_G=\{\alpha,\beta\}$, and $\Phi_G^+=\{\alpha,\beta,\beta+\alpha,\beta+2\alpha,\beta+3\alpha,2\beta+3\alpha\}$. Of these, $\alpha,\beta+\alpha$ and $\beta+2\alpha$ are short, while $\beta, \beta+3\alpha$ and $2\beta+3\alpha$ are long. We have $\langle\alpha,\check{\beta}\rangle=-1$ and $\langle\beta,\check{\alpha}\rangle=-3$ (and of course $\langle\alpha,\check{\alpha}\rangle=\langle\beta,\check{\beta}\rangle=2$), also, $\rho_G=5\alpha+3\beta$.

In this section we omit $\alpha$, so $\Delta_M=\{\beta\}$ and it is known that $M\simeq\GL_2$. We have $\Phi_N=\Phi_G^+-\{\beta\}$, $\rho_P=5\alpha+\frac{5}{2}\beta$, $\langle\rho_P,\check{\alpha}\rangle=5/2$ and $\tilde{\alpha}=2\alpha+\beta$.
Let $f$ be a cuspidal Hecke eigenform of weight $k$ and trivial character, with $a_p(f)=p^{(k-1)/2}(\alpha_p+\alpha_p^{-1})$, and let $\Pi$ be the corresponding unitary, cuspidal, automorphic representation of $\GL_2\simeq M$. Then we have $\lambda=\left(\frac{k-1}{2}\right)\beta$ and $\chi_p=-\log_p(\alpha_p)\beta$ (at any unramified prime $p$).
\vskip10pt
\begin{tabular}{|c|c|c|c|}\hline $\gamma\in\Phi_N$ & $\check{\gamma}$ & $\langle\lambda+s\tilde{\alpha},\check{\gamma}\rangle$ & $|\chi_p(\check{\gamma}(p))|_p$\\\hline $\alpha$ & $\check{\alpha}$ & $-\frac{3}{2}(k-1)+s$ & $\alpha_p^{-3}$\\$\beta+\alpha$ & $\check{\alpha}+3\check{\beta}$ & $\frac{3}{2}(k-1)+s$ & $\alpha_p^3$\\$\beta+3\alpha$ & $\check{\alpha}+\check{\beta}$ & $-\frac{1}{2}(k-1)+s$ & $\alpha_p^{-1}$\\ $2\beta+3\alpha$ & $\check{\alpha}+2\check{\beta}$ & $\frac{1}{2}(k-1)+s$ & $\alpha_p$ \\$\beta+2\alpha$ & $2\check{\alpha}+3\check{\beta}$ & $2s$ & $1$\\\hline \end{tabular}
\vskip10pt
Using the table, $L_{\Sigma}(s,\Pi,r_1)=L_{\Sigma}(\Sym^3 f,s+\frac{3}{2}(k-1))$ and $L_{\Sigma}(s,\Pi,r_2)=\zeta_{\Sigma}(s)$. For $\lambda+s\tilde{\alpha}$ to be algebraically integral, we need $s\in \frac{1}{2}+\ZZ$ and
$$\lambda+s\tilde{\alpha}=\left(\frac{k-1}{2}\right)\beta+s(\beta+2\alpha).$$
We have $L(1+s,\Pi,r_1)$ critical for $0<s<\frac{k-1}{2}$ with $s\in \frac{1}{2}+\ZZ$, in which case $L(1+2s,\Pi,r_2)$ is also critical.
Choose $w:\beta\mapsto 2\beta+3\alpha$ and $2\alpha+\beta\mapsto\alpha$. This is a rotation clockwise through $\pi/3$. Then
$$w(\lambda+s\tilde{\alpha})=\left(\frac{k-1}{2}\right)(2\beta+3\alpha)+s\alpha=(k-1)\beta+((3/2)(k-1)+s)\alpha,$$
which is strictly dominant, since $\langle w(\lambda+s\tilde{\alpha}),\check{\alpha}\rangle=2s>0$ and $\langle w(\lambda+s\tilde{\alpha}),\check{\beta}\rangle=\frac{1}{2}(k-1)-s>0$. In fact $w(\lambda+s\tilde{\alpha})=k_1\omega_1+k_2\omega_2$ with $k_1=2s$, $k_2=\frac{1}{2}(k-1)-s$, $\omega_1=2\alpha+\beta$ and $\omega_2=3\alpha+2\beta$ the fundamental dominant weights.

Now $\hat{G_2}=G_2$, and its irreducible $7$-dimensional representation has weights $\mu$ as in the table below. Recall that $\chi_p=-\log_p(\alpha_p)\beta$.
\vskip10pt
\begin{tabular}{|c|c|}\hline $\mu$ & $|(\chi_p+s\tilde{\alpha})(\mu(p))|_p$ \\ \hline $\check{\alpha}+2\check{\beta}$ & $\alpha_pp^{-s}$ \\$-(\check{\alpha}+2\check{\beta})$ & $\alpha_p^{-1}p^s$\\ $\check{\alpha}+\check{\beta}$ & $\alpha_p^{-1}p^{-s}$\\$-(\check{\alpha}+\check{\beta})$ & $\alpha_pp^s$\\$\check{\beta}$ & $\alpha_p^2$\\$-\check{\beta}$ & $\alpha_p^{-2}$\\ $0$ & $1$\\\hline \end{tabular}
\vskip10pt
The trace is
$$t:=(\alpha_p+\alpha_p^{-1})(p^{-s}+p^s)+(\alpha_p+\alpha_p^{-1})^2-1.$$
Given $\q$ such that $q>3k-2$ and
$$\ord_{\q}\left(\frac{L_{\Sigma}(\Sym^3 f,1+s+(3/2)(k-1))}{\Omega}\right)>0,$$ or $q>2+2s$ and
$$\ord_{\q}\left(\frac{\zeta_{\Sigma}(1+2s)}{\pi^{1+2s}}\right)>0,$$
let $\tilde{\Pi}$ be the automorphic representation of $G(\A)$ conjectured to exist by Conjecture \ref{main}.
According to Gross and Savin, there should be a functorial lift $\tilde{\Pi}'$ to $\GSp_3(\A)$, with $(a+3,b+2,c+1)=(k_1+2k_2,k_1+k_2,k_2)$ (\cite[Introduction]{GS}), which in our case is $(k-1,\frac{k-1}{2}+s,\frac{k-1}{2}-s)$. (Notice that $a=b+c$ for such a lift.) If we take the Satake parameter for $\tilde{\Pi}'$ at an unramified prime $p$, plug it into the $8$-dimensional spinor representation and take the trace, then it follows from the first equation in \cite[\S 2,(1.8)]{GS} that we get (at least when $q\neq p$) something congruent to $1+t=(\alpha_p+\alpha_p^{-1})(p^{-s}+p^s)+(\alpha_p+\alpha_p^{-1})^2$. Multiplying by $p^{(a+b+c+6)/2}=p^{k-1}$, we find that, in the notation of \S 8, $T_{f_1+f_2+f_3+f_0}(\tilde{\Pi}')$ should be congruent modulo $\q$ to $a_p(f)(a_p(f)+p^{b+2}+p^{c+1})$. We see that $\tilde{\Pi}'$ would satisfy the conjecture in Case 2 of \S 8, in the special case $f=g$. Four of the examples in \cite[Table 3]{BFvdG1} are like this. Note that $L(\Sym^2(f)\otimes f,s)=L(\Sym^3 f,s)L(f,s-(k-1))$, $1+s+(3/2)(k-1)=a+b+6$ and $1+2s=b-c+2$.

\section{Example: $G=G_2$, omitting the long root.}
In this section, in a departure from the usual notation, we omit $\beta$, so $\Delta_M=\{\alpha\}$ and again it is known that $M\simeq\GL_2$. Then $\Phi_N=\Phi_G^+-\{\alpha\}$, $\rho_P=(9/2)\alpha+3\beta$, $\langle\rho_P,\check{\beta}\rangle=3/2$ and $\tilde{\beta}=2\beta+3\alpha$.
Let $f$ be a cuspidal Hecke eigenform of weight $k'$ and trivial character, with $a_p(f)=p^{(k'-1)/2}(\alpha_p+\alpha_p^{-1})$, and let $\Pi$ be the corresponding unitary, cuspidal, automorphic representation of $\GL_2\simeq M$. Then we have $\lambda= ((k'-1)/2)\alpha $ and $\chi_p=-\log_p(\alpha_p)\alpha$ (at any unramified prime $p$).
\vskip10pt
\begin{tabular}{|c|c|c|c|}\hline $\gamma\in\Phi_N$ & $\check{\gamma}$ & $\langle\lambda+s\tilde{\beta},\check{\gamma}\rangle$ & $|\chi_p(\check{\gamma}(p))|_p$\\\hline $\beta$ & $\check{\beta}$ & $-\frac{k'-1}{2}+s$ & $\alpha_p^{-1}$\\$\beta+3\alpha$ & $\check{\alpha}+\check{\beta}$ & $\frac{k'-1}{2}+s$ & $\alpha_p$\\$2\beta+3\alpha$ & $\check{\alpha}+2\check{\beta}$ & $2s$ & $1$\\$\beta+\alpha$ & $\check{\alpha}+3\check{\beta}$ & $-\frac{k'-1}{2}+3s$ & $\alpha_p^{-1}$\\$\beta+2\alpha$ & $2\check{\alpha}+3\check{\beta}$ & $\frac{k'-1}{2}+3s$ & $\alpha_p$\\\hline \end{tabular}
\vskip10pt
Using the table, $L_{\Sigma}(s,\Pi,r_1)=L_{\Sigma}(s,\Pi,r_3)=L_{\Sigma}(f,s+\frac{k'-1}{2})$ and $L_{\Sigma}(s,\Pi,r_2)=\zeta_{\Sigma}(s)$. For $\lambda+s\tilde{\alpha}$ to be algebraically integral, we need $s\in \frac{1}{2}+\ZZ$ and
$$\lambda+s\tilde{\beta}=\left(\frac{k'-1}{2}\right)\alpha+s(2\beta+3\alpha).$$
We have $L(1+s,\Pi,r_1)$ critical for $0<s<\frac{k'-1}{2}$ (with $s\in \frac{1}{2}+\ZZ$), in which case $L(1+2s,\Pi,r_2)$ is also critical, but to get $L(1+3s,\Pi,r_3)$ critical too, we need $0<s<\frac{k'-1}{6}$.
Choose $w:\alpha\mapsto \beta+2\alpha$ and $3\alpha+2\beta\mapsto\beta$. This is a rotation anticlockwise through $\pi/3$. Then
$$w(\lambda+s\tilde{\beta})=\left(\frac{k'-1}{2}\right)(\beta+2\alpha)+s\beta=(k-1)\alpha+((1/2)(k'-1)+s)\beta,$$
which is strictly dominant, since $\langle w(\lambda+s\tilde{\beta}),\check{\alpha}\rangle=\frac{1}{2}(k'-1)-3s>0$ and $\langle w(\lambda+s\tilde{\beta}),\check{\beta}\rangle=2s>0$. In fact $w(\lambda+s\tilde{\alpha})=k_1\omega_1+k_2\omega_2$ with $k_1=\frac{1}{2}(k'-1)-3s$, $k_2=2s$.

This time, with $\chi_p=-\log_p(\alpha_p)\alpha$, we get the following table.
\vskip10pt
\begin{tabular}{|c|c|}\hline $\mu$ & $|(\chi_p+s\tilde{\beta})(\mu(p))|_p$ \\ \hline $\check{\alpha}+2\check{\beta}$ & $p^{-2s}$  \\$-(\check{\alpha}+2\check{\beta})$ & $p^{2s}$\\ $\check{\alpha}+\check{\beta}$ & $\alpha_pp^{-s}$\\$-(\check{\alpha}+\check{\beta})$ & $\alpha_p^{-1}p^s$ \\$\check{\beta}$ & $\alpha_p^{-1}p^{-s}$\\$-\check{\beta}$ & $\alpha_pp^s$\\ $0$ & $1$ \\\hline \end{tabular}
\vskip10pt
The trace is
$$t:=(p^{-s}+p^s)(\alpha_p+\alpha_p^{-1}+p^{-s}+p^s)-1.$$
This time $(a+3,b+2,c+1)=(k_1+2k_2,k_1+k_2,k_2)=(\frac{1}{2}(k'-1)+s, \frac{1}{2}(k'-1)-s, 2s)$ (so $c$ must be even).
In the same manner as the previous section, if (case $i=1$),
$$\ord_{\q}\left(\frac{L(f,a+4)}{\Omega^{\pm}}\right)>0$$
with $q>k'$, or (case $i=2$)
$$\ord_{\q}\left(\frac{\zeta_{\Sigma}(c+2)}{\pi^{c+2}}\right)>0$$
with $q>c+3$, or (case $i=3$),
$$\ord_{\q}\left(\frac{L(f,a+c+5)}{\Omega^{\pm}}\right)>0$$
with $q>k'$, then we expect $\tilde{\Pi}'$ for $\GSp_3(\A)$, with
$$T_{f_1+f_2+f_3+f_0}(\tilde{\Pi}')\equiv p^{\frac{k-1}{2}+s}(1+t)=(1+p^{c+1})(a_p(f)+p^{b+2}+p^{a+3})\pmod{\q}.$$
Note that $k'=a+b+6$.

First we look at $i=1$. If we choose $j,k$ such that $b+2=k-2$ and $a+3=j+k-1$, then
$a_p(f)+p^{b+2}+p^{a+3}=a_p(f)+p^{k-2}+p^{j+k-1}$, and moreover, since $a=b+c$ we have $a+4=j+k$. So according to \S 7 (i.e. Harder's conjecture)
there should be a cuspidal Hecke eigenform $F$ of genus $2$, weight $\Sym^j\otimes\det^k$, with $T(p)(F)\equiv a_p(f)+p^{b+2}+p^{a+3}\pmod{\q}$.
Then the congruence with right-hand-side $(1+p^{c+1})(a_p(f)+p^{b+2}+p^{a+3})$ becomes an instance of that in \S 9, and the argument for why $\q$ should divide $\frac{L(j+2,F,\st)}{\Omega}$ is the same, connected with \cite[Conjecture 5.3]{DIK}. (Note that the condition $j\leq k-4$ follows from $c\leq b$.)

Next consider the case $i=2$. Recalling the case $G=\GL_2$, there should be a Hecke eigenform $g$, of weight $\ell=c+2$, satisfying $a_p(g)\equiv 1+p^{c+1}\pmod{\q}$ for $p\notin\Sigma$. Then $\tilde{\Pi}'$ would satisfy the $i=2$ case of the congruence in Case 1 of \S 8, with right hand side $a_p(g)(a_p(f)+p^{b+2}+p^{a+3})$. (Note that $c+2=a-b+2$, because $a=b+c$.) Looking at it yet another way (as might be suggested by the previous paragraph and the case $i=2$ in \S 7), $(1+p^{c+1})(a_p(f)+p^{b+2}+p^{a+3})\equiv (1+p^{c+1})(a_p(f)+p^{b+2}a_p(g))\pmod{\q}$, so we have the type of congruence discussed at the end of \S 9, but this time the first factor in the product $L(s,\Pi',\st)=\zeta(s)L(f\otimes g,s+a+3)$ is the relevant one.

An $i=3$ example where we expect a congruence is $(a,b,c)=(10,8,2)$, $k'=24$, $q=179$ (with level $1$, $\Sigma=\emptyset$). The space of genus $3$ cusp forms for level $1$ and $(a,b,c)=(10,8,2)$ is $1$-dimensional, and we checked the congruence for $p\leq 17$, using Hecke eigenvalues calculated as in \cite{BFvdG1}. Moreover, it appears that we can relax the condition $a=b+c$ and expect the same kind of congruence to hold. We found congruences in the following examples (checked for $p\leq 17$). Again, in each case the space of genus $3$ cusp forms is $1$-dimensional. The divisibility of $\frac{L(f,a+c+5)}{\Omega^{\pm}}$ may be checked in almost all cases using the final table of \cite{vdG}. (Strictly speaking, in those cases where the field of coefficients is not rational, we did not check that the congruence and the divisibility involve the same divisor of $q$, except in the case $q=179$.) For the $q=43$ example we used the computer package Magma instead, and for those cases where $q\not >k'$ (i.e. $q=19$ and $q=37$) the divisibility is not really well-defined, so we did not check, though it may be significant that in the normalisation in \cite{vdG}, the $19$ does occur as a factor.
\vskip10pt
\begin{tabular}{|c|c|c|c|}\hline $(a,b,c)$ & $k'=a+b+6$ & $a+c+5$ & $q$\\ \hline $(12,6,2)$ & $24$ & $19$ & $73$ \\$(10,8,2)$ & $24$ & $17$ & $179$\\ $(16,4,2)$ & $26$ & $23$ & $43$\\ $(14,6,2)$ & $26$ & $21$ & $97$\\ $(12,8,2)$ & $26$ & $19$ & $29$\\ $(10,8,4)$ & $24$ & $19$ & $73$\\ $(13,9,2)$ & $28$ & $20$ & $157$\\ $(13,7,4)$ & $26$ & $22$ & $19$\\ $(10,10,4)$ & $26$ & $19$ & $29$\\$(21,3,2)$ & $30$ & $28$ & $97$\\$(12,8,6)$ & $26$ & $23$ & $43$\\$(10,10,6)$ & $26$ & $21$ & $97$\\$(10,10,8)$ & $26$ & $23$ & $43$\\$(13,11,8)$ & $30$ & $26$ & $593$\\$(13,11,10)$ & $30$ & $28$ & $97$\\$(15,13,12)$ & $34$ & $32$ & $103$\\$(16,16,16)$ & $38$ & $37$ & $37$\\\hline \end{tabular}
\vskip10pt
Notice that there are three triples of $(k',a+c+5,q)$ that appear twice in the table, namely $(24,19,73)$, $(26,21,97)$ and $(30,28,97)$, and $(26,23,43)$ even appears three times.

How are we to account for these congruences when $a\neq b+c$? One way is to use the non-maximal parabolic subgroup $P$ of $\GSp_3$ with Levi subgroup $M\simeq\GL_2\times\GL_1\times\GL_1$ and $\Delta_M=\{e_1-e_2\}$. One uses now two parameters, $s$ and $t$, but replaces $N$ by the intersection of the unipotent radicals of the maximal parabolics containing $P$. Playing this game with non-maximal parabolics is sometimes fruitful, but seems in many cases to produce congruences that do not hold. Further details are omitted here.

In \cite{BFvdG1}, the various congruences involving Siegel cusp forms of level $1$ are viewed as being between Hecke eigenspaces of, on the one hand, parts of the inner cohomology of a Siegel modular variety (with coefficients in a local system) coming from cusp forms, and on the other, ``Eisenstein'' or ``endoscopic'' contributions to the Euler characteristic. They do not actually calculate the Hecke actions on the latter, but for each congruence there is a contribution to the Euler characteristic whose form suggests that it accounts for the congruence.
The above congruences should be accounted for by an ``Eisenstein contribution" ``$s_{a+b+6}L^{b+2}$'' (see \cite[Conjecture 7.12]{BFvdG1}), where $s_{a+b+6}$ is the dimension of the space of cusp forms of weight $a+b+6$ and $L$ is the Lefschetz motive. This term was accidentally overlooked in \cite{BFvdG1}.

Formally, the simplest way to arrive at the congruences is to substitute an Eisenstein series of weight $c+2$ for the cusp form $g$ in the congruence in Case 1 of \S 8.

\section{Cohomology of local systems.}
Useful references are \cite[2.3]{H3},\cite[1.2]{H4}, \cite[2.2,2.3]{H6}. Let
$\MMM$, a finite-dimensional vector space over $\QQ$, be the space
of an irreducible rational representation of $G$, with highest weight $\Lambda_G$.
Let $K_{\infty}$ be the product inside $G(\RR)$ of the identity components of its maximal compact subgroup and of $Z(G)(\RR)$. Let $K_f$ be an open
compact subgroup of $G(\A_f)$ (where $\A_f$ is the finite adeles).
Let $\SSS_{K_f}=G(\QQ)\backslash G(\A)/K_{\infty}K_f$. This is a
finite union of quotients $\Gamma^{(g_f)}\backslash X_{\infty}$ of
the symmetric domain $X_{\infty}=G(\RR)/K_{\infty}$ by the discrete
groups $\Gamma^{(g_f)}:=G(\QQ)\cap g_f K_fg_f^{-1}$, where $g_f$ is
a set of representatives for the finite set of double cosets
$G(\QQ)\backslash G(\A_f)/K_f$. On $\SSS_{K_f}$ there is a locally
constant sheaf $\MMM$ associated to the representation of the same
name. The direct limit over smaller and smaller $K_f$,
$H^{\bullet}(\SSS,\MMM):=\varinjlim_{K_f}H^{\bullet}(\SSS_{K_f},\MMM)$,
is naturally a $G(\A_f)$-module.

At each level there is a Borel-Serre compactification
$\overline{\SSS_{K_f}}$, a manifold with corners. The sheaf $\MMM$
can naturally be extended to $\overline{\SSS_{K_f}}$ (on which it
has the same cohomology) and restricted to the boundary $\partial
S_{K_f}:=\overline{\SSS_{K_f}}-{\SSS_{K_f}}$. This induces maps
$H^{\bullet}(\SSS_{K_f},\MMM)\rightarrow H^{\bullet}(\partial
\SSS_{K_f},\MMM)$ and $H^{\bullet}(\SSS,\MMM)\rightarrow
H^{\bullet}(\partial \SSS,\MMM)$. The kernel is the ``inner''
cohomology $H^{\bullet}_{!}(\SSS,\MMM)$.

At each level the boundary $\partial \SSS_{K_f}$ is stratified by
submanifolds labelled by conjugacy classes of parabolic subgroups
defined over $\QQ$. By considering the restriction of the sheaf
$\MMM$ to such a boundary stratum, and taking a limit over $K_f$,
one obtains, in particular for a maximal parabolic subgroup $P$, a map
$$I_{P(\A_f)}^{G(\A_f)}H^{\bullet}(\SSS^M,
H^{\bullet}(\mathfrak{n},\MMM))\rightarrow
H^{\bullet}(\partial\SSS,\MMM)$$
of $G(\A_f)$-modules. (This is
ordinary, not unitary, parabolic induction.) Here $\mathfrak{n}$
is the Lie algebra of the unipotent radical of $P$, and the Lie
algebra cohomology is viewed as a representation of the Levi
quotient $M$. The cohomology $H^{\bullet}(\SSS^M,\_\_ )$ is defined
to be a direct limit, over open compact subgroups $C_f$ of
$M(\A_f)$, of $H^{\bullet}(\SSS^M_{C_f},\_\_ )$, where
$\SSS^M_{C_f}=M(\QQ)\backslash M(\A)/K^M_{\infty}C_f$, and
$K^M_{\infty}$ is the intersection of $K_{\infty}$ with $M(\RR)$.

In fact there is an injection (see just before 2.3.2 of \cite{H3})
$$I_{P(\A_f)}^{G(\A_f)}H^{\bullet}_{\mathrm{cusp}}\bigl(\SSS^M,
H^{\bullet}(\mathfrak{n},\MMM)\bigr)\hookrightarrow
H^{\bullet}(\partial\SSS,\MMM)$$ of $G(\A_f)$-modules, where the
cuspidal subspace is analytically defined. Using a theorem of Kostant, we
get a decomposition
$$H^{q}(\mathfrak{n},\MMM)\simeq \bigoplus_{w'\in W^P,\, \ell(w')=q}
F_{\mu_{w'}}$$ of representations of $M$. Here, if $W$ is the Weyl
group of $G$, and $W_M$ the Weyl group of $M$, then $W^P$ is the set
of Kostant representatives of $W_M\backslash W$, i.e. $\{w'\in
W : {w'}^{-1}(\Delta_M)\subset\Phi_G^+\}$, and $\ell(w')$
is the length of $w'$.
We define $\mu_{w'}=w'.\Lambda_G=w'(\Lambda_G+\rho)-\rho$, which happens to
be a highest weight for $M$, and $F_{\mu_{w'}}$ is the irreducible
rational representation of $M$ with that highest weight, so
$$\bigoplus_{w'\in W^P}\Ind_{P(\A_f)}^{G(\A_f)}\,\,H^{q-\ell(w)}_{\mathrm{cusp}}(\SSS^M,
F_{\mu_w})\hookrightarrow H^{q}(\partial\SSS,\MMM).$$

Now consider cuspidal automorphic representations $\Pi$ of $M(\A)$ and $\tilde{\Pi}$ of $G(\A)$, as in Conjecture \ref{main}, with restrictions $\Pi_f$ (note that this notation means something different than above)
to $M(\A_f)$ and $\tilde{\Pi}_f$ to $G(\A_f)$, respectively. In Harder's set-up, one starts not with $\Pi$ but with $\tilde{\Pi}$, and where $\tilde{\Pi}_f$ occurs in $H^q_!(\SSS,\MMM)$, for some $q$ and some $\Lambda_G$, while $\Ind_P^G(\Pi\otimes|s\tilde{\alpha}|)_f$ occurs in $I_{P(\A_f)}^{G(\A_f)}\,\,H^{q-\ell(w')}_{\mathrm{cusp}}(\SSS^M,
F_{\mu_{w'}})$, for some $w'$. (Knowing the possible degrees of cohomology that are relevant, one experiments with $w'$ of the correct length, then works back from $\mu_{w'}$ to deduce the type of $\Pi_{\infty}$.) Harder's congruence is then
between $G(\A_f)$-modules occurring in $H^q_!(\SSS,\MMM)$ and in $H^{q}(\partial\SSS,\MMM)$, in fact it ought to arise from fusion between Hecke modules with integral coefficients, if one uses an integral model for $\MMM$. Fixing the Kostant representative $w'$, let $\Lambda_M:=\mu_{w'}$. Next we try to work out how Harder's approach relates to ours.

First, by Wigner's Lemma \cite[9.4.1]{Wa}, the dual of $\MMM$ has the same infinitesimal character as $\tilde{\Pi}_{\infty}$, so
$$-w_0^G\Lambda_G+\rho_G=w(\lambda+s\tilde{\alpha}).$$
Similarly, the infinitesimal character of the component at $\infty$ of a representation of $M(\A)$ whose finite part occurs in $H^{q-\ell(w')}_{\mathrm{cusp}}(\SSS^M,
F_{\mu_{w'}})$ is $-w_0^M\Lambda_M+\rho_M$. Subtracting $\rho_P$ to take into account the difference between ordinary and unitary induction, we must have $$-w_0^M\Lambda_M+\rho_M-\rho_P=\lambda+s\tilde{\alpha}=w^{-1}(-w_0^G\Lambda_G+\rho_G)=-w^{-1}w_0^G(\Lambda_G+\rho_G).$$
It is then $(\Pi\otimes|s\tilde{\alpha}-\rho_P|)_f=(\Pi\otimes|(s-\langle\rho_P,\check{\alpha}\rangle)\tilde{\alpha}|)_f$ that occurs in $H^{q-\ell(w')}_{\mathrm{cusp}}(\SSS^M,
F_{\mu_{w'}})$.
Applying $-w_0^M$ to both sides,
$$\Lambda_M+\rho_M+\rho_P=w_0^Mw^{-1}w_0^G(\Lambda_G+\rho_G).$$
Since $\rho_M+\rho_P=\rho_G$, this becomes
$$\Lambda_M=w_0^Mw^{-1}w_0^G(\Lambda_G+\rho_G)-\rho_G.$$
Comparing with $\Lambda_M=w'(\Lambda_G+\rho_G)-\rho_G$, we must have the relation
$$w'=w_0^Mw^{-1}w_0^G, \text{ i.e. } w=w_0^G(w')^{-1}w_0^M,$$
between Harder's Kostant representative $w'$ and our $w$, which was chosen so that $w(\lambda+s\tilde{\alpha})$ is strictly dominant for small $s>0$.

The requirement that $\ell(w')=q_G-q_M$, where $q_G$ and $q_M$ are such that
$H^{q_G}_{\mathrm{cusp}}(\SSS,\MMM)$ and $H^{q_M}_{\mathrm{cusp}}(\SSS^M,F_{\mu_{w'}})$ can be non-zero, must tell us something about our $w$.
\begin{lem}\label{length}
If $w'=w_0^Mw^{-1}w_0^G$ then $\ell(w')$ is the number of positive roots in $w(\Phi_N)$.
\end{lem}
\begin{proof} The length of an element of $W_G$ is the number of elements of $\Phi_G^+$ that it sends to $-\Phi_G^+$.
$$w'=(-w_0^M)w^{-1}(-w_0^G).$$
Now $-w_0^G$ maps $\Phi_G^+$ to $\Phi_G^+$, while $-w_0^M$ maps $\Phi_M^+$ to $\Phi_M^+$, while sending $\Phi_N$ (the remainder of $\Phi_G^+$) to $-\Phi_G^+$. So $\ell(w')$ is the number of elements of $\Phi_G^+$ mapped to $\Phi_N$ by $w^{-1}$, hence the lemma.
\end{proof}

So it should be the case that the number of positive roots in $w(\Phi_N)$ is $q_G-q_M$. It is easy to check this directly in many cases, including some examples below.
According to a theorem of Li and Schwermer \cite[3.5]{LS}, as long as $\Lambda_G$ is regular, $H^q_{\mathrm{cusp}}(\SSS,\MMM)$ can be non-zero at most for $q\in[q_0(G),q_0(G)+\ell_0(G)]$, where $\ell_0(G)=\rank(G)-\rank(K_G)$ ($K_G$ being a maximal compact subgroup of $G(\RR)$) and $q_0(G)+(1/2)\ell_0(G)=(1/2)\dim X_G$, with $X_G=G(\RR)/K_G$. In other words, $q_G$ must lie inside this interval, and is uniquely determined in cases where $\ell_0(G)=0$.

In \S\S 6--9 we could have used the semi-simple group $G=\PGSp_g$, for which $K_G\simeq U(g)$, with $\rank(G)=\rank(K_G)=g$, so $\ell_0(G)=0$. By counting generators of Lie algebras, $\dim G(\RR)=g+g(g-1)+g(g+1)=2g^2+g$, while $\dim K_G=g^2$, so $\dim X_G=g(g+1)$ and $q_G=\frac{g(g+1)}{2}$.
Looking instead at $G=\GL_n$, with $K_G\simeq O(n)$, we have $\ell_0(G)=n-[n/2]$ and $\dim X_G=n^2-\frac{n(n-1)}{2}=\frac{n(n+1)}{2}$, so
$q_0(G)=\begin{cases} \frac{n^2}{4} \text{ if $n$ is even};\\ \frac{n^2-1}{4} \text{ if $n$ is odd.}\end{cases}$.

{\bf Example 1.} Let $G=\PGSp_g$, $\alpha=e_1-e_2$. Then $M\simeq \GL_1\times\PGSp_{g-1}$. If $\lambda=a_2e_2+\cdots a_ge_g$ with $a_2>\cdots>a_g>0$ (so regular and dominant for $M$) then to make $w(\lambda)$ dominant for $G$ we choose $w\in W_G$ such that $e_g\mapsto e_{g-1}\mapsto\cdots\mapsto e_2\mapsto e_1\mapsto e_g$. Then $w(\lambda)=a_2e_1+\cdots+a_ge_{g-1}$. (We omit terms in $e_0$, as if we were dealing with $\Sp_g$.) Now $\Phi_N=\{e_1-e_j: 1<j\leq g\}\cup\{e_1+e_j:1\leq j\leq g\}$. All of the $w(e_1+e_j)$ are in $\Phi_G^+$, but none of the $w(e_1-e_j)$, so $\ell(w')=g$. Subtracting this from $q_G=\frac{g(g+1)}{2}$ gives $\frac{(g-1)g}{2}$, which is $q_0(M)$.

{\bf Example 2.} Let $G=\Sp_g$, $\alpha=2e_g$. Then $M\simeq\GL_g$, $\Phi_M^+=\{e_i-e_j:1\leq i<j\leq g\}$ and $\Phi_N=\{e_i+e_j:1\leq i\leq j\leq g\}$. Purely for notational convenience, suppose that $g$ is even. If $\lambda=a_1(e_1-e_g)+a_2(e_2-e_{g-1})+\ldots +a_{g/2}(e_{g/2}-e_{(g/2)+1})$,
with $a_1>a_2>\ldots>a_{g/2}>0$,
(this is regular, self-dual and dominant for $M$),
then we choose $w\in W_G$ such that $(e_1,-e_g, e_2,-e_{g-1},\ldots,e_{g/2},-e_{(g/2)+1})\mapsto(e_1, e_2,\ldots,e_g)$. One checks that $w(e_i+e_j)$, with $i\leq j$, is in $\Phi_G^+$ precisely for
$1\leq i\leq j\leq \frac{g}{2}$ and for $1\leq i\leq \frac{g}{2}$ with $\frac{g}{2}+1\leq j\leq g+1-i$. Hence $\ell(w')=\frac{g(g+2)}{8}+\frac{g(g+2)}{8}=\frac{g(g+2)}{4}$. Subtracting this from $q_G=\frac{g(g+1)}{2}$ gives $\frac{g^2}{4}$, which is $q_0(M)$.

\section{The Bloch-Kato Conjecture.}
Given a congruence mod $\q$ as in Conjecture \ref{main}, our goal in this section is to explain how the existence of a non-zero element in $H^1_{\Sigma}(\Q,T_{i,\q}^*\otimes (E_{\q}/O_{\q}))$ (cf. Conjecture \ref{BK}) might follow. Thus the Bloch-Kato conjecture motivates a belief in Conjecture \ref{main}, via the support the former would receive from the latter.

Let $\theta_{\mu}:\hat{G}\rightarrow\GL_n$, for some $n$ (and some $\mu$), be the rational representation of highest weight $\mu$. This is a morphism of algebraic groups defined over $\QQ$. If $\lambda+s\tilde{\alpha}\in X^*(T)$ then, according to \cite[Conjecture 3.2.1]{BG}, there should exist a representation $\rho_{\tilde{\Pi}}:\Gal(\Qbar/\QQ)\rightarrow \hat{G}(E_{\q})$, such that if $p\notin\Sigma\cup\{q\}$ then
$\rho_{\tilde{\Pi}}(\Frob_p^{-1})$ is conjugate in $\hat{G}(E_{\q})$ to the Satake parameter $t(\chi(\tilde{\pi}_p))$. Letting $\omega:\Gal(\Qbar/\QQ)\rightarrow \ZZ_q^{\times}$ be the cyclotomic character, and $c(\mu):=\langle w(\lambda+s\tilde{\alpha}),\mu\rangle$, let $\rho_{\mu,\tilde{\Pi}}:=\omega^{-c(\mu)}\theta_{\mu}\circ\rho_{\tilde{\Pi}}:\Gal(\Qbar/\QQ)\rightarrow\GL_n(E_{\q})$.
Then for $p\notin\Sigma\cup\{q\}$, $\Tr(\rho_{\mu,\tilde{\Pi}}(\Frob_p^{-1}))=p^{c(\mu)}\Tr(\theta_{\mu}(t(\chi(\tilde{\pi}_p))))$, and $\rho_{\mu,\tilde{\Pi}}(\Gal(\Qbar/\QQ))\subseteq \theta_{\mu}(\hat{G}(E_{\q}))$. Even if $\lambda+s\tilde{\alpha}\notin X^*(T)$, $\mu$ should induce a functorial lift of $\tilde{\Pi}$ to $\Pi'$ say, on $\GL_n(\A)$, and there should be a Galois representation associated to $\Pi'\otimes|\det|^{c(\mu)}$, which we also call $\rho_{\mu,\tilde{\Pi}}:\Gal(\Qbar/\QQ)\rightarrow\GL_n(E_{\q})$, with again $\Tr(\rho_{\mu,\tilde{\Pi}}(\Frob_p^{-1}))=p^{c(\mu)}\Tr(\theta_{\mu}(t(\chi(\tilde{\pi}_p))))$. It seems reasonable to suppose that again we will have $\rho_{\mu,\tilde{\Pi}}(\Gal(\Qbar/\QQ))\subseteq \theta_{\mu}(\hat{G}(E_{\q}))$.

Since $\Gal(\Qbar/\QQ)$ is compact, there must be an invariant $O_{\q}$-lattice in $E_{\q}^n$, so conjugating by some element of $\GL_n(E_{\q})$, we may suppose that $\rho_{\mu,\tilde{\Pi}}(\Gal(\Qbar/\QQ))\subseteq \GL_n(O_{\q})$. We may adjust $\theta_{\mu}$ (now defined over $E_{\q}$) to preserve the property $\rho_{\mu,\tilde{\Pi}}(\Gal(\Qbar/\QQ))\subseteq \theta_{\mu}(\hat{G}(E_{\q}))$. In fact there is an integral structure on $\hat{G}$ such that $\hat{G}(O_{\q})=\theta_{\mu}^{-1}(\GL_n(O_{\q}))$, so now $\rho_{\mu,\tilde{\Pi}}(\Gal(\Qbar/\QQ))\subseteq \theta_{\mu}(\hat{G}(O_{\q}))$. (For simplicity, let's imagine that this integral structure is the Chevalley group scheme.) Identifying $\hat{G}$ with $\theta_{\mu}(\hat{G})$, we have $\rho_{\mu,\tilde{\Pi}}:\Gal(\Qbar/\QQ)\rightarrow \hat{G}(O_{\q})$, and its reduction $\rhobar_{\mu,\tilde{\Pi}}:\Gal(\Qbar/\QQ)\rightarrow \hat{G}(\FF_{\q})$.

Similarly, identifying $\hat{M}$ with $\theta_{\mu}(\hat{M})$ (same $\theta_{\mu}$ as above, though its restriction to $\hat{M}$ might not be irreducible), we get $\rho_{\mu,\Pi\otimes|s\tilde{\alpha}|}:\Gal(\Qbar/\QQ)\rightarrow \hat{M}(O_{\q})$, and its reduction $\rhobar_{\mu,\Pi\otimes|s\tilde{\alpha}|}:\Gal(\Qbar/\QQ)\rightarrow \hat{M}(\FF_{\q})$, with $\Tr(\rho_{\mu,\Pi\otimes|s\tilde{\alpha}|}(\Frob_p^{-1}))=p^{c(\mu)}\Tr(\theta_{\mu}(t(\chi_p+s\tilde{\alpha})))$ ($\chi_p$ as in \S 2).
We shall impose a condition that the image $\rhobar_{\mu,\Pi\otimes|s\tilde{\alpha}|}(\Gal(\Qbar/\QQ))$ is not contained in any proper parabolic subgroup of $\hat{M}(\FF_{\q})$.

Applying Conjecture \ref{main} to $T_{\mu}$ if $\mu$ is minuscule, or more generally to the sum of $T_{\mu}$ and a suitable integral linear combination of $T_{\mu'}$ for $\mu'<\mu$ (see \cite[(3.12)]{Gr}), it predicts that, for all $p\notin\Sigma\cup\{q\}$,
$$\Tr(\rhobar_{\mu,\Pi\otimes|s\tilde{\alpha}|}(\Frob_p^{-1}))=\Tr(\rhobar_{\mu,\tilde{\Pi}}(\Frob_p^{-1})) \text{ in $\FF_{\q}$}.$$
It seems reasonable to suppose that a consequence of this is that (for $q$ sufficiently large), after conjugating $\rho_{\mu,\tilde{\Pi}}$ by something in $\hat{G}(O_{\q})$ before we reduce, we can get $\rhobar_{\mu,\tilde{\Pi}}(\Gal(\Qbar/\QQ))\subseteq \hat{P}_{\op}(\FF_q)$, with the projection of $\rhobar_{\mu,\tilde{\Pi}}$, from $\hat{P}_{\op}(\FF_q)$ to its Levi subgroup $\hat{M}(\FF_{\q})$, equal to $\rhobar_{\mu,\Pi\otimes|s\tilde{\alpha}|}$. (Here, $\hat{P}_{\op}$ is the parabolic subgroup opposite to $\hat{P}$, with unipotent radical $\hat{N}_{\op}$ such that $\Phi_{\hat{N}_{\op}}=-\Phi_{\hat{N}}$.) This is easy to prove in the case $G=\GL_{n}, M=\GL_{n/2}\times\GL_{n/2}$, with $n$ even, $\theta_{\mu}$ the identity and $q>2n$, using the Brauer-Nesbitt theorem. In this case our condition on $\rhobar_{\mu,\Pi\otimes|s\tilde{\alpha}|}(\Gal(\Qbar/\QQ))$ amounts to the irreducibility of the two representations to $\GL_{n/2}(\FF_{\q})$, which are therefore the composition factors of $\rhobar_{\mu,\tilde{\Pi}}$.

From now on, we abbreviate $\rho_{\mu,\tilde{\Pi}}$, $\rhobar_{\mu,\tilde{\Pi}}$, $\rho_{\mu,\Pi\otimes|s\tilde{\alpha}|}$ and $\rhobar_{\mu,\Pi\otimes|s\tilde{\alpha}|}$ to $\rhot$, $\rhobart$, $\rho$ and $\rhobar$, respectively. For $j\geq 1$, let $\hat{N}_{\op}(\FF_{\q})^{(j)}$ be the $j^{\mathrm{th}}$ derived subgroup, so $\hat{N}_{\op}(\FF_{\q})/\hat{N}_{\op}(\FF_{\q})^{(1)}$ is the abelianisation. Define
$$C:\Gal(\Qbar/\QQ)\rightarrow \hat{N}_{\op}(\FF_{\q})/\hat{N}_{\op}(\FF_{\q})^{(1)}
\text{ by }
C(g)=[\rhobart(g)\rhobar(g)^{-1}],$$
where $[\cdot]$ denotes the class in the quotient group. (That $\rhobart(g)\rhobar(g)^{-1}\in\hat{N}_{\op}(\FF_{\q})$ follows from the supposition in the previous paragraph.) We have that
$$C(gh)=\rhobart(gh)\rhobar(gh)^{-1}=\rhobart(g)\rhobar(g)^{-1}\rhobar(g)\rhobart(h)\rhobar(h)^{-1}\rhobar(g)^{-1}=C(g)\ad(\rhobar(g))(C(h)).$$
Therefore $C$ is a cocycle, representing a Galois cohomology class
denoted $C^{(0)}$ in $H^1(\QQ,\hat{N}_{\op}(\FF_{\q})/\hat{N}_{\op}(\FF_{\q})^{(1)})$, where the action of $\Gal(\Qbar/\QQ)$ on $\hat{N}_{\op}(\FF_{\q})/\hat{N}_{\op}(\FF_{\q})^{(1)}$ is via $\ad(\rhobar)$.

We would like this class to be non-zero. It might not be, but in the case $G=\GL_{n}, M=\GL_{n/2}\times\GL_{n/2}$, a trivial modification of an argument of Ribet (the case $n=2$, \cite[Proposition 2.1]{Ri}) produces a class that is non-zero, assuming the irreducibility of $\rhot$.
Henceforth, in place of this irreducibility condition, we assume that $\rhot(\Gal(\Qbar/\QQ))$ is not contained in any proper parabolic subgroup of $\hat{G}(E_{\q})$. This ought to be true except in cases where $\tilde{\Pi}$ is some kind of endoscopic lift. In general to get a non-zero class, especially when $m>1$ (in $r=\oplus_{i=1}^m r_i$), it seems to be necessary to use a somewhat different argument, such as the following.

Suppose that $C^{(0)}=0$ (since if $C^{(0)}\neq 0$ then we already have what we want). Then $C$ is a coboundary, so for some $n\in\hat{N}_{\op}(\FF_{\q})$, and every $g\in\Gal(\Qbar/\QQ)$, $\rhobart(g)\rhobar(g)^{-1}=[\rhobar(g)n\rhobar(g)^{-1}n^{-1}]$, so
$$n\rhobart(g)n^{-1}\rhobar(g)^{-1}=[n(\rhobar(g)n\rhobar(g)^{-1})n^{-1}(\rhobar(g)n\rhobar(g)^{-1})^{-1}]=[1].$$
Now $n\rhobart(g)n^{-1}\in\hat{P}_{\op}(\FF_{\q})$ and has the same projection to $\hat{M}(\FF_{\q})$ as $\rhobart(g)$. The condition that this projection is $\rhobar(g)$ only determines $\rhobart(g)$ up to conjugation by $\hat{N}_{\op}(\FF_{\q})$, and $n\rhobart(g)n^{-1}$ is an equally good choice. Making this the new $\rhobart(g)$, $\rhobart(g)\rhobar(g)^{-1}$ now takes values in $\hat{N}_{\op}(\FF_{\q})^{(1)}$, so we may start again and use it to define a cocycle representing a cohomology class $C^{(1)}\in H^1(\QQ,\hat{N}_{\op}(\FF_{\q})^{(1)}/\hat{N}_{\op}(\FF_{\q})^{(2)})$.
If this is $0$, similarly we get $C^{(2)}\in H^1(\QQ,\hat{N}_{\op}(\FF_{\q})^{(2)}/\hat{N}_{\op}(\FF_{\q})^{(3)})$, etc.

Considering root subgroups, it is easy to see that
$$\oplus_{j\geq 0}(\hat{N}_{\op}(\FF_{\q})^{(j)}/\hat{N}_{\op}(\FF_{\q})^{(j+1)}) \simeq \oplus_{i\geq 1}\hat{\n}_{\op,i},$$
where $\hat{\n}_{\op,i}$ is defined in the same manner as $\hat{\n}_i$ in \S 3. If some $C^{(j)}\neq 0$, we have therefore, for some $i$, a non-zero element of $H^1(\QQ, \hat{\n}_{\op,i})$ (with the adjoint action of $\rhobar$), which is our interim goal.

If all the $C^{(j)}=0$, for all $j\geq 0$, then (after conjugation in $\hat{N}_{\op}(\FF_{\q})$) $\rhobart$ takes values in $\hat{M}(\FF_{\q})$.
(Of course, $\hat{N}_{\op}(\FF_{\q})^{(j)}=\{1\}$ eventually.) Equivalently, after conjugation in $\hat{N}_{\op}(O_{\q})$, $\rhot$ takes values in $\hat{M}_1$, where for $k\geq 1$ we define
$$\hat{M}_k:=\{h\in\hat{G}(O_{\q}): h\in \hat{P}_{\op}\pmod{\q} \text{ and }h\in \hat{P}\pmod{\q^k}\}.$$
Let also $\hat{N}_{\op,k}:=\{h\in\hat{M}_k: h\pmod{\q}=1\}$, so, (after conjugation in $\hat{N}_{\op}(\FF_{\q})$), $\rhot(g)\rho(g)^{-1}\in \hat{N}_{\op,1}$, for all $g\in\Gal(\Qbar/\QQ)$.

Now we may play the same game, substituting $\hat{N}_{\op,1}/\hat{N}_{\op,2}$ for $\hat{N}_{\op}(\FF_{\q})$. When $N$ is not abelian, it might not be the case that $\hat{N}_{\op,k}/\hat{N}_{\op,k+1}\simeq \hat{N}_{\op}(\FF_{\q})$. But using the fact that $\hat{N}_{\op,k}/\hat{N}_{\op,k+1}$ is generated by the images of elements of the form $\exp(q^kn'+qm+qn)$, with $n'$,$m$ and $n$ integral elements of $\hat{\n}_{\op}$, $\hat{\mm}$ and $\hat{\n}$ respectively, with respect to a Chevalley basis, one sees that it shares with $\hat{N}_{\op}(\FF_{\q})$ the property
$$\oplus_{j\geq 0}(\hat{N}_{\op,k}/\hat{N}_{\op,k+1})^{(j)}/(\hat{N}_{\op,k}/\hat{N}_{\op,k+1})^{(j+1)}\simeq \oplus_{i\geq 1}\hat{\n}_{\op,i}.$$
Starting with $k=1$ if we are unable to produce a non-zero class using $\hat{N}_{\op,k}/\hat{N}_{\op,k+1}$ in place of $\hat{N}_{\op}(\FF_{\q})$ then, after conjugation in $\hat{N}_{\op}(O_{\q})$, $\rhot$ takes values in $\hat{M}_{k+1}$. If this fails for all $k\geq 1$ then, after conjugation in $\hat{N}_{\op}(O_{\q})$, $\rhot$ takes values in $\cap_{k\geq 1}\hat{M}_{k}=\hat{P}(O_{\q}),$ contrary to our ``irreducibility'' hypothesis.

Thus we must get, for some $i$, a non-zero element of $H^1(\QQ, \hat{\n}_{\op,i})$. The Killing form $(n',n)\mapsto\tr(\ad(n')\ad(n))$ puts $\hat{\n}_{\op,i}$ and $\hat{\n}_i$ in perfect duality (with $\Phi_N^i$ and $-\Phi_N^i$ as dual bases). It respects the adjoint action of $\hat{M}$, showing that $\hat{\n}_{\op,i}$ and $\hat{\n}_i$ are dual as representations of $\hat{M}$. It follows that $\hat{\n}_{\op,i}$ injects into $T_{i,\q}^*\otimes (E_{\q}/O_{\q})$ (as its $\q$-torsion subgroup), at least for some choice of lattice $T_{i,\q}$, but see the remark at the end of \S 4. Thus we get a non-zero class in $H^1(\Q,T_{i,\q}^*\otimes (E_{\q}/O_{\q}))$, as long as $H^0(\Q,\hat{\n}_{\op,i})$ is trivial. This class satisfies the Bloch-Kato local condition at any prime $p\notin\Sigma\cup\{q\}$, since $\rhot$ is unramified at such a prime. We should expect it also to satisfy the local condition at $q$, given that $\tilde\Pi_q$ is unramified and all the $\theta_{\mu}\circ\rhot$ ought to be crystalline at $q$. Thus we expect to have a non-zero element of
$H^1_{\Sigma}(\Q,T_{i,\q}^*\otimes (E_{\q}/O_{\q}))$, which would fit with
$$\ord_{\q}\left(\frac{L_{\Sigma}(1+is,\Pi,r_i)}{\Omega}\right)>0$$
and the Bloch-Kato conjecture, {\em if it were the same $i$}, but we might have no way of knowing that when $m>1$.

Note that in the above ``construction'', we could switch the roles of $P$ and $P_{\op}$, thus producing a non-zero class in $H^1_{\Sigma}(\Q,T_{i',\q}\otimes (E_{\q}/O_{\q}))$ instead (for some $i'$). This seems like a problem, since in general it will not be the case that also
$$\ord_{\q}\left(\frac{L_{\Sigma}(1-i's,\Pi,r_{i'})}{\Omega'}\right)>0$$ which is what the Bloch-Kato conjecture would suggest if this is a critical value. Now $1-i's$ is paired with $i's$, which differs in parity from the original $1+i's$, so a way out would be if it is always the case that for some $i'$ there is a parity condition stopping $L_{\Sigma}(1-i's,\Pi,r_{i'})$ from being a critical value. For a non-critical value, $L(1-i's,\Pi,r_{i'})$ should be $0$, and (if not near-central) its order of vanishing should be the rank of a $\q$-adic Selmer group. This is an analogue of the rank part of the Birch and Swinnerton-Dyer conjecture, see the ``conjectures'' $C_r(M)$ and $C_{\lambda}^i(M)$ in \cite[\S 1,\S6.5]{Fo}. Reducing $\pmod{\q}$, one would expect something non-zero in $H^1_{\Sigma}(\Q,T_{i',\q}\otimes (E_{\q}/O_{\q}))$, making it no problem to have constructed such a thing. There will be a parity condition if $\langle\lambda,\check{\gamma}\rangle=0$ for some $\gamma\in\Phi_N^{i'}$, and in the examples we looked at in the earlier sections, this always happens for some $i'$, as one can see by examining the tables. Recalling Lemma \ref{symmetry}, it would suffice to show that there is some $\gamma\in\Phi_N$ such that $w_0^M\gamma=\gamma$.

\end{document}